\documentclass[a4paper,11pt]{amsart}
\pdfoutput=1   
\usepackage[utf8]{inputenc}
\usepackage[english]{babel}
\usepackage{indentfirst}
\usepackage{csquotes}

\usepackage{graphicx}
\usepackage{multicol}
\usepackage{multirow}

\usepackage{amsmath}
\usepackage{amsfonts}
\usepackage{amssymb}
\usepackage{amsthm}
\usepackage{mathrsfs}

\usepackage[usenames,dvipsnames,svgnames,table]{xcolor}

\usepackage{tikz}
\usetikzlibrary{positioning}
\usetikzlibrary{arrows.meta}

\usepackage[style=numeric,sorting=nyt,maxbibnames=7,maxcitenames=5,block=space,backend=bibtex]{biblatex}
 \makeindex

\makeatletter
\renewcommand*{\@textcolor}[3]{%
  \protect\leavevmode
  \begingroup
    \color#1{#2}#3%
  \endgroup
}
\makeatother

\usepackage[bookmarks=true]{hyperref}
\definecolor{azulmaneiro}{rgb}{0, 0.125, 0.666}
\definecolor{brownmaneiro}{rgb}{0.4, 0.2, 0}
\hypersetup{colorlinks=true, linkcolor=azulmaneiro, citecolor=brownmaneiro, urlcolor=red}

\newtheorem{thm}{Theorem}[section]
\newtheorem{conj}[thm]{Conjecture}
\newtheorem{cor}[thm]{Corollary}
\newtheorem{lem}[thm]{Lemma}
\newtheorem{pps}[thm]{Proposition}

\theoremstyle{definition}

\newtheorem{rmkx}[thm]{Remark}
\newenvironment{rmk}
  {\pushQED{\qed}\rmkx}
  {\popQED\endrmkx}
\newtheorem{examplex}[thm]{Example}
\newenvironment{exm}
  {\pushQED{\qed}\examplex}
  {\popQED\endexamplex}
	
	\numberwithin{equation}{section}

\newcommand{\N}{\mathbb{N}}
\newcommand{\Z}{\mathbb{Z}}
\newcommand{\Nzero}{\mathbb{Z}_{\geq 0}}
\newcommand{\Q}{\mathbb{Q}}
\newcommand{\F}{\mathbb{F}}
\newcommand{\R}{\mathbb{R}}
\newcommand{\K}{\mathbb{K}}
\newcommand{\laurent}[1]{#1[t,t^{-1}]}
\newcommand{\OS}{\mathcal{O}_S}

\DeclareMathOperator{\GL}{GL}

\DeclareMathOperator{\SL}{SL}
\DeclareMathOperator{\Sp}{Sp}

\newcommand{\Mult}{\mathbb{G}_m}
\newcommand{\Addi}{\mathbb{G}_a}
\newcommand{\uCD}{\mathcal{G}_\Phi^{{\rm sc}}}

\newcommand{\Bzero}{\mathbf{B}_2^{\circ}}
\newcommand{\BzeroR}{\mathbf{B}_2^{\circ}(R)}
\newcommand{\mcP}{\mathcal{P}}
\newcommand{\mcB}{\mathcal{B}}
\newcommand{\mcG}{\mathcal{G}}
\newcommand{\mcH}{\mathcal{H}}
\newcommand{\mcU}{\mathcal{U}}
\newcommand{\mfX}{\mathfrak{X}}
\newcommand{\mbP}{\mathbf{P}}
\newcommand{\mbB}{\mathbf{B}}
\newcommand{\mbG}{\mathbf{G}}
\newcommand{\mbD}{\mathbf{D}}
\newcommand{\mbT}{\mathbf{T}}
\newcommand{\mbA}{\mathbf{A}}
\newcommand{\mbU}{\mathbf{U}}
\newcommand{\PmbB}{\mathbb{P}\mathbf{B}}
\newcommand{\Aff}{\mathbb{A}\mathrm{ff}} 

\newcommand{\phee}{\varphi}

\newcommand{\set}[1]{\left\{ #1 \right\}}
\newcommand{\spans}[1]{\langle {#1} \rangle}
\newcommand{\leer}{\varnothing}

\newcommand{\mb}{\mathbf}
\newcommand{\mc}{\mathcal}
\newcommand{\mf}{\mathfrak}

\newcommand{\scrH}{\mathscr{H}}

\newcommand{\into}{\hookrightarrow}
\newcommand{\onto}{\twoheadrightarrow}

\newcommand{\Fn}[1]{\mathrm{F}_{#1}} 
\newcommand{\FPn}[1]{\mathrm{FP}_{#1}}

\newcommand{\eij}{e_{ij}}
\newcommand{\Ekl}[1]{{\rm E}_{#1}}
\newcommand{\ekl}[1]{e_{#1}}

\newcommand{\doubleref}[2]{\ref{#1}(\ref{#2})}

\newcommand{\tA}{\text{A}}
\newcommand{\tB}{\text{B}}
\newcommand{\tC}{\text{C}}
\newcommand{\tD}{\text{D}}

\newcommand{\tG}{\text{G}_2}

\DeclareMathOperator{\Hom}{Hom}
\DeclareMathOperator{\rk}{rk}

\DeclareMathOperator{\carac}{char}
\DeclareMathOperator{\Diag}{Diag}
\DeclareMathOperator{\Frac}{Frac}

\def\customdate{\empty}
\renewcommand{\date}[1]{\def\customdate{#1}}

\title[Finiteness lengths of soluble matrix groups]{On the finiteness length of some soluble linear groups}

\author[Y. Santos Rego]{Yuri Santos Rego}
\address{Universit\"at Bielefeld, \newline 
Fakult{\"a}t f{\"u}r Mathematik, \newline 
Postfach 100131, \newline 
D-33501 Deutschland \newline}
\curraddr{Otto-von-Guericke-Universit\"at Magdeburg, \newline 
Fakult\"at f\"ur Mathematik -- Institut f\"ur Algebra und Geometrie, \newline 
Postfach 4120, 39016 Magdeburg, Deutschland \newline}
\email{yuri.santos@ovgu.de}
\thanks{The work was supported by the Deutscher Akademischer Austauschdienst (\emph{F\"order-ID} 57129429) and the Bielefelder Nachwuchsfonds.}

\begin{document}

\thispagestyle{empty}

\begin{abstract}
Given a commutative unital ring $R$, we show that the finiteness length of a group $G$ is bounded above by the finiteness length of the Borel subgroup of rank one $\mathbf{B}_2^\circ(R)=\left( \begin{smallmatrix} * & * \\ 0 & * \end{smallmatrix} \right)\leq\SL_2(R)$ whenever $G$ admits certain $R$-representations with metabelian image. Combined with results due to Bestvina--Eskin--Wortman and Gandini, this gives a new proof of (a generalization of) Bux's equality on the finiteness length of $S$-arithmetic Borel groups. 
We also give an alternative proof of an unpublished theorem due to Strebel, characterizing finite presentability of Abels' groups $\mathbf{A}_n(R) \leq \GL_n(R)$ in terms of $n$ and $\mathbf{B}_2^\circ(R)$. This generalizes earlier results due to Remeslennikov, Holz, Lyul'ko, Cornulier--Tessera, and points out to a conjecture about the finiteness length of such groups.
\end{abstract}

\maketitle

\vspace{-10pt}

\section{Introduction}

The finiteness length $\phi(G)$ of a group $G$ is the supremum of the $n \in \Nzero$ for which $G$ admits an Eilenberg--MacLane space $K(G,1)$ with compact $n$ skeleton. The finiteness length is a quasi-isometry invariant~\cite{AlonsoFPnQI} which can be interpreted as a tool measuring `how finite' $G$ is, from the topological point of view. For instance, if $G$ is either finite or of homotopical type $\mathrm{F}$---the latter meaning that $G$ admits a compact $K(G,1)$---then $\phi(G) = \infty$. Also, $G$ is finitely presented if and only if $\phi(G) \geq 2$; see e.g.~\cite{Ratcliffe}.

Throughout, $R$ always denotes a commutative ring with unity. By a \emph{classical group} we mean an affine $\Z$-group scheme $\mcG$ which is either $\GL_n$ for some $n \geq 2$ or a universal Chevalley--Demazure group scheme $\uCD$, such as $\mcG_{\tA_{n-1}}^{\rm sc} = \SL_n$ or $\mcG_{\tC_{n}}^{\rm sc} = {\rm Sp}_{2n}$; refer e.g. to~\cite{VavilovPlotkin} for an extensive survey on Chevalley--Demazure groups over arbitrary rings. In this work we are interested in the groups of $R$-points of certain soluble subgroups of $\mcG$. An important role will be played by the Borel subgroup of rank one,
\[
 \BzeroR = \left( \begin{smallmatrix} * & * \\ 0 & * \end{smallmatrix} \right) \leq \SL_2(R).
\]
We shall also consider the \emph{affine groups} $\Aff \cong \left( \begin{smallmatrix} * & * \\ 0 & 1 \end{smallmatrix} \right) \leq \GL_2$ and $\Aff_- \cong \left( \begin{smallmatrix} 1 & * \\ 0 & * \end{smallmatrix} \right) \leq \GL_2$.

The purpose of this paper is twofold. First, using algebraic methods, we give an upper bound on the finiteness length of groups with certain soluble $R$-representations, including many parabolic subgroups of classical groups. Second, we record the current state of knowledge on the finiteness lengths of H.~Abels' soluble matrix groups $\set{\mbA_n(R)}_{n \geq 2}$. This includes a different, topological proof of an unpublished result of R.~Strebel, who classified which $\mbA_n(R)$ are finitely presented. Our main results are the following.

\begin{thm} \label{apendice}
Let $\mcG$ be either classical or one of the affine groups $\Aff$, $\Aff_-$. Let also $\mfX$ and $\mcH$ denote, respectively, a unipotent root subgroup and a maximal torus of $\mcG$. (In the affine cases, $\mfX \rtimes \mcH$ is just the whole group scheme.) Suppose a group $G$ admits a representation $\rho : G \to \mcG(R)$ whose image is $\mathrm{Im(\rho)} = \mfX(R) \rtimes \mcH(R)$ and such that the sequence $\ker(\rho) \into G \onto \mathrm{Im}(\rho)$ splits. Then $\phi(G) \leq \phi(\BzeroR)$.
\end{thm}

\begin{thm} \label{Strebel's}
 If $R$ is \emph{not} finitely generated as a ring, then $\phi(\mbA_n(R)) = 0$ for all $n \geq 2$.  Otherwise, the following hold.
 \begin{enumerate}
  \item \label{facil1} 
  $\phi(\mbA_2(R)) > 0$ if and only if $R$ is finitely generated \emph{as a} $\Z$\emph{-module}, in which case $\phi(\mbA_n(R)) = \phi(\BzeroR) = \infty$ for all $n \geq 2$.
    \item \label{facil2} 
    If $R$ is \emph{infinitely generated} as a $\Z$-module, then $\phi(\mbA_3(R)) =~\min\set{1, \phi(\BzeroR)}$.
    \item \label{nao-facil} 
  If $n \geq 4$ and $R$ is \emph{infinitely generated} as a $\Z$-module, then $\phi(\mbA_n(R))\leq\phi(\BzeroR)$ and, given $\ell \in \set{1,2}$, one has $\phi(\mbA_n(R)) \geq \ell$ whenever $\phi(\BzeroR) \geq \ell$.  
   \end{enumerate}
\end{thm}

This article is organized as follows. In sections~\ref{almostretracts} and~\ref{Abels'gps} below we motivate the main theorems~\ref{apendice} and~\ref{Strebel's}, respectively. Section~\ref{almostretracts} includes a new proof of K.-U.~Bux's main result in~\cite{Bux04}, and in section~\ref{Abels'gps} we state a conjecture about the finiteness lengths of Abels' groups. We recall in section~\ref{prelim} some standard facts to be used throughout. Theorems~\ref{apendice} and~\ref{Strebel's} are proved in sections~\ref{capum} and~\ref{capdois}, respectively. 

\subsection{Theorem~\ref{apendice} -- motivation and examples} \label{almostretracts} The problem of determining the finiteness length of an $S$-arithmetic group is an ongoing challenge; see, for instance, the introductions of \cite{Bux04,BuxWortman07,BKWRank} for an overview. 

Suppose $\mcB$ is a Borel subgroup of a Chevalley--Demazure group and let $\OS$ be a Dedekind ring of arithmetic type and positive characteristic, such as $\F_q[t]$ or $\F_q[t,t^{-1},(t-1)^{-1}]$. In~\cite{Bux04}, K.-U.~Bux showed that the $S$-arithmetic group $\mcB(\OS)$ satisfies the equality $\phi(\mcB(\OS)) = \vert S\vert - 1$. The first step of his geometric proof~\cite[Theorem~5.1]{Bux04} yields the upper bound $\phi(\mcB(\OS)) \leq \vert S\vert - 1$. 
The number $\vert S\vert - 1$, in turn, had already been observed to equal $\phi(\Bzero(\OS))$~\cite[Corollary~3.5]{Bux04}. Our inspiration for Theorem~\ref{apendice} was to give a simple algebraic explanation for the inequality $\phi(\mcB(\OS)) \leq \phi(\Bzero(\OS))$ since this would likely extend to larger classes of rings. And this is in fact the case; see Section~\ref{capum} for the proof of Theorem~\ref{apendice}. We also expect Theorem~\ref{apendice} to have appropriate analoga in the context of Kac--Moody groups. 

In the arithmetic set-up, we combine Theorem~\ref{apendice} with results due to M.~Bestvina, A.~Eskin and K.~Wortman \cite{BEW} and G.~Gandini~\cite{Gandini} to obtain the following proof of (a generalization of) Bux's equality.

\begin{thm}
 \label{Bux's}
 Let $\mbP$ be a proper parabolic subgroup of a non-commutative, connected, reductive, split linear algebraic group $\mbG$ defined over a global field $\K$. Denote by $\mbU_\mbP$ the unipotent radical of $\mbP$ and by $\mb{T}_\mbP$ a maximal torus of $\mbG$ contained in $\mbP$. For any $S$-arithmetic subgroup $\Gamma \leq \mbU_\mbP \rtimes \mb{T}_\mbP$, the following inequalities hold.
 \[
  \vert S\vert - 1 \leq \phi(\Gamma) \leq \phi(\Bzero(\OS)).
 \]
Moreover, if $\K$ has positive characteristic, then $\phi(\Gamma) = \vert S\vert - 1$.
\end{thm}

\begin{proof}
 Standard arguments allow us to assume, without loss of generality, that $\mbG$ is classical; see e.g. the steps in~\cite[Section~2.6(c)]{Behr98} and pass over to parabolic subgroups---notice that Satz~1 cited by Behr holds regardless of characteristic. Since $S$-arithmetic subgroups of a given linear algebraic group are commensurable, we may further restrict ourselves to the (now well-defined) group of $\OS$-points $\mbU_\mbP(\OS) \rtimes \mbT_\mbP(\OS) \leq \mbG(\OS)$. In the set-up above, the first inequality follows from~\cite[Theorem~6]{BEW} and Lemma~\ref{obviousboundsonphi} since $\Gamma$ is an extension of $\mbU_\mbP(\OS) \rtimes \mb{A}_\mbP(\OS)$ by a finitely generated abelian group (due to Dirichlet's unit theorem), where $\mb{A}_\mbP$ is the maximal torus in the center of the Levi factor of $\mbP$. The second inequality is a direct consequence of our Theorem~\ref{apendice}.

For the last claim, suppose $\carac(\K)> 0$. Since $\Bzero(\OS) \supseteq \left( \begin{smallmatrix} 1 & * \\ 0 & 1 \end{smallmatrix} \right) \cong (\OS,+)$, one has that $\Bzero(\OS)$ has no bounds on the orders of its finite subgroups because $\OS$ contains infinite dimensional vector spaces over the prime field of $\K$~\cite[Section~23]{O'Meara}. But $\Bzero(\OS)$ acts by cell-permuting homeomorphisms on the product of $\vert S \vert$ Bruhat--Tits trees, each such tree being associated to the locally compact group $\SL_2(\Frac(\OS)_v)$ attached to the place $v \in S$; cf.~\cite{SerreTrees}. Since the stabilizers of this action are finite~\cite[Section~3.3]{BruhatTits}, it follows that $\Bzero(\OS)$ belongs to P.~Kropholler's $\mb{H}\mf{F}$ class and thus Gandini's theorem~\cite{Gandini} applies, yielding $\phi(\Bzero(\OS)) \leq \vert S \vert - 1$.
\end{proof}

An alternative proof of the inequality $\phi(\Gamma) \leq \vert S\vert -1$ (positive characteristic) was announced by K.~Wortman~\cite[p.~2169]{BEW}. While giving a stronger version of Bux's result, Theorem~\ref{Bux's} furnishes further examples of non-metabelian soluble linear groups with prescribed finiteness properties. To the best of our knowledge, the only known cases were Abels--Witzel's groups (see Section~\ref{ExemploConjectura} and~\cite{StefanAbels}) and Bux's own examples~\cite{Bux04}. 

It is natural to wonder which properties of $S$-arithmetic subgroups, or of associated symmetric spaces and buildings, are intrinsic to the underlying algebraic group and thus independent of the field of definition. (Compare e.g. cohomological finiteness properties~\cite{BoSe,BKWRank} and distortion dimension~\cite{BuxWortman07,BEW} in the semi-simple case.) For our groups, although the work of Tiemeyer~\cite[Theorem~4.3 and (the proof of)~Theorem~4.4]{Tiemeyer} implies the stronger result $\phi(\Gamma) = \infty$ when $\K$ has characteristic zero, Theorem~\ref{Bux's} is advantageous in the sense that its inequalities are \emph{uniform}, i.e. independent of characteristic.  
 
Examples of non-soluble groups that fit into the framework of Theorems~\ref{apendice} and~\ref{Bux's} arise, for instance, from parabolic subgroups of classical groups. 
 
 \begin{exm} \label{oexemplo}
   Recall that the (standard) parabolic subgroups of $\SL_n$ over a field $\K$ are the subgroups of block (upper or lower) triangular matrices. Pictorially, a parabolic $\mcP \leq \SL_n$ is e.g. of the form
  \[
\mcP \cong
  \begin{pmatrix}
 \mb{n_1 \times n_1} & \multicolumn{1}{|c}{*} & \cdots & * \\ \cline{1-2}
 0 & \multicolumn{1}{|c|}{\mb{n_2 \times n_2}} & \ddots & \vdots \\ \cline{2-3}
 \vdots & \ddots & \multicolumn{1}{|c|}{\ddots} & * \\ \cline{3-4}
 0 & \cdots & \multicolumn{1}{c|}{0} & \mb{n_k \times n_k}
\end{pmatrix},
\]
where $k \leq n$ is the number of blocks. In particular, $\mcP$ contains the diagonal subgroup---the standard maximal torus---of $\SL_n$, and each $\mb{n_i \times n_i}$ block on the diagonal with $n_i > 1$ is isomorphic to $\SL_{n_i}$. Now suppose $\K$ is a global function field, i.e. a finite extension of $\F_p(t)$ for some prime $p$. If there exists an index $j < k$ for which $n_j = 1 = n_{j+1}$, then Theorem~\ref{Bux's} yields $\phi(\Gamma) < \vert S\vert$ for any $S$-arithmetic subgroup $\Gamma \leq \mcP$.
 \end{exm}
 
Computing the finiteness lengths of $S$-arithmetic parabolics in positive characteristic is an open problem. Example~\ref{oexemplo} provides a new result in this direction. 
 
 Going back to arbitrary base rings, we stress that results on the finiteness length of \emph{non}-$S$-arithmetic discrete linear groups are scarce. The most prominent examples were obtained by Bux--Mohammadi--Wortman~\cite{BuxMohammadiWortman} and Gandini~\cite[Corollary~4.1]{Gandini} using Bruhat--Tits buildings, and by Kropholler--Mullaney~\cite{KrophollerMullaney} building upon works of \AA{}berg and Groves--Kochloukova.
 
As shown, the class of groups to which Theorem~\ref{apendice} applies is quite large. One can notably summarize many such examples via the so-called groups of type (R). These were studied by M.~Demazure and A.~Grothendieck in the 1960s and generalize parabolic subgroups of reductive affine group schemes; see \cite[Expos{\'e}~XXII, Chapitre~5]{SGA3.3}.

\begin{cor} \label{type(R)}
 Let $\mbG$ be an affine group scheme defined over $\Z$ and let $\mb{H} \leq \mcG$ be a $\Z$-subgroup, of type \emph{(R)} with soluble geometric fibers, of a classical group $\mcG$. If there exists a $\Z$-retract $r : \mbG \to \mb{H}$, then $\phi(\mbG(R)) \leq \phi(\BzeroR)$ for every commutative ring $R$ with unity.
\end{cor}

\begin{proof}
 This follows immediately from Theorem~\ref{apendice} and~\cite[Expos{\'e}~XXII, Proposition~5.6.1 and Corollaire~5.6.5]{SGA3.3}.
\end{proof}

The following gives a concrete series of finitely generated, non-amenable, non-arithmetic groups with prescribed upper bounds on the finiteness length. Other such examples include (Stallings--Bieri--)Bestvina--Brady's groups and certain members of the family of generalized Thompson groups; cf.~\cite{RachelStefanMatt} and references therein for an overview.

 \begin{exm} \label{outroexemplo}
Similarly to Example~\ref{oexemplo}, let $\mcP$ be a parabolic $\Z$-subscheme of $\SL_n$ with $k$ diagonal blocks of sizes $n_j$, with at least one block of size at least two and with at least one index $j < k$ for which $n_j = 1 = n_{j+1}$. For instance, we can take $\mcP$ to be
\[ \left( \begin{smallmatrix} * & * & * & * \\ * & * & * & * \\ 0 & 0 & * & * \\ 0 & 0 & 0 & * \end{smallmatrix} \right) \leq \SL_4 \, \mbox{ or } \, \left( \begin{smallmatrix} * & * & * & * & * & * & * \\ * & * & * & * & * & * & * \\ * & * & * & * & * & * & * \\ 0 & 0 & 0 & * & * & * & * \\ 0 & 0 & 0 & 0 & * & * & * \\ 0 & 0 & 0 & 0 & 0 & * & * \\ 0 & 0 & 0 & 0 & 0 & * & * \end{smallmatrix} \right) \leq \SL_7. \]
 Following Kropholler--Mullaney~\cite{KrophollerMullaney}, fix $\ell \in \N$ and let $R_\ell$ be the integral domain \[R_\ell = \Z[x,x^{-1},(1+x)^{-1},\ldots,(\ell+x)^{-1},\frac{1}{\ell!}] \subset \Q(x).\] 
The groups $\mcP(R_\ell)$ are non-amenable because they contain $\SL_2(\Z)$, and it is an exercise (e.g. using the relations from section~\ref{prelim}) to check that they are finitely generated. Corollary~\ref{type(R)} gives $\phi(\mcP(R_\ell)) < \ell + 2$ since $\phi(\Bzero(R_\ell)) = \ell + 1$ by~\cite{Desi0,KrophollerMullaney}.
 \end{exm}

We remark that Theorem~\ref{apendice} admits a slight geometric modification by weakening the hypothesis on the map $G \to \mf{X}(R) \rtimes \mcH(R)$ at the cost of an extra assumption on the base ring $R$; see Section~\ref{RporQR} for details.

\subsection{Theorem~\ref{Strebel's} -- the group schemes of Herbert Abels} \label{Abels'gps} For every natural number $n \geq 2$, consider the following $\Z$-group scheme.
\[
 \mbA_n :=
 \left( \begin{smallmatrix}
  1 & * & \cdots & \cdots & * \\
  0 & * & \ddots & & \vdots \\
  \vdots & \ddots & \ddots & \ddots & \vdots \\
  0 & \cdots & 0 & * & * \\
  0 & \cdots & \cdots & 0 & 1
 \end{smallmatrix} \right)
 \leq \GL_n.
\]

Interest in the infinite family $\set{\mbA_n}_{n \geq 2}$, nowadays known as Abels' groups, was sparked in the late 1970s when Herbert Abels~\cite{Abels0} published a proof of finite presentability of the group $\mbA_4(\Z[1/p])$, where $p$ is any prime number; see also~\cite[Sections~0.2.7 and~0.2.14]{Abels}.  Abels' groups emerged as counterexamples to long-standing problems in group theory---see, for instance, \cite[Proposition~A5]{BenGriHar}---and later became a source to construct groups with peculiar properties; cf.~\cite{dC-Guyot-PitschIsolated,Carrion-Dadarlat-Eckhardt,Bi-dC-G-S,BeLuTho} for recent examples. 

Regarding their finiteness lengths, not long after Abels announced that $\phi(\mbA_4(\Z[1/p])) \geq 2$, Ralph Strebel went on to generalize this in the handwritten notes~\cite{StrebelAbels}, which never got to be published and only came to our attention after Theorem~\ref{Strebel's} was established. He actually gives necessary and sufficient conditions for Abels--Strebel's groups, defined by
\[
A_n(R,Q) = \{g \in \mbA_n(R) \mid g \mbox{ has diagonal entries in } Q \leq R^\times\},
\]
to have $\phi(A_n(R,Q)) \geq 2$. (Note that $A_n(R,R^\times) = \mbA_n(R)$.)  
It had been announced by Remeslennikov~\cite[p.~210]{Abels0} that $\mbA_4(\Z[x,x^{-1},(x+1)^{-1}])$ admits a finite presentation---a similar example is treated in detail in Strebel's manuscript. 
Shortly thereafter, Bieri~\cite{BieriZHomology} observed that $\phi(\mbA_n(\Z[1/p])) \leq n-2$.

In the mid 1980s, S.~Holz and A.~N.~Lyul'ko proved independently that $\mbA_n(\Z[1/p])$ and $\mbA_n(\Z[1/m])$, respectively, are finitely presented as well---for all $n, m, p \in \N$ with $n \geq 4$ and $p$ prime~\cite[Anhang]{Holz},~\cite{Lyulko}. Their techniques differ from Strebel's in that they consider large subgroups of $\mbA_n$ and relations among them to check for finite presentability of the overgroup. Holz~\cite{Holz} 
pushed the theory further by giving the first example of a soluble, non-metabelian group of finiteness length exactly three, namely $\mbA_5(\Z[1/p])$. 
Abels--Brown~\cite{AbelsBrown} later showed that $\phi(\mbA_n(\Z[1/p])) \geq n-2$. (This actually holds in greater generality; see Section~\ref{ExemploConjectura}.) In \cite{StefanAbels}, S.~Witzel generalizes the family $\set{\mbA_n}_{n\geq 2}$ and proves that such groups over $\Z[1/p]$ for $p$ an odd prime have, in addition, varying Bredon finiteness properties.

Besides the above examples in characteristic zero and Strebel's manuscript, the only published case of a finitely presented Abels' group over a torsion ring is also $S$-arithmetic. In~\cite{dC-T}, Y.~de~Cornulier and R.~Tessera prove, among other things, that $\phi(\mbA_4(\F_p[t,t^{-1},(t-1)^{-1}])) \geq 2$. They remark that $\phi(\mbA_4(\F_p[t])) = 0$ and $\phi(\mbA_4(\laurent{\F_p})) = 1$, and point out that similar results, including the one on finite presentability over $\F_p[t,t^{-1},(t-1)^{-1}]$, should hold for $n \geq 4$ with mechanical changes~\cite[Remark~5.5]{dC-T}.

As far as generators and relations are concerned, Theorem~\ref{Strebel's} generalizes to arbitrary rings the above mentioned results on presentations of Abels' groups. The previous discussion also indicates the following natural problem.

\begin{conj} \label{EndlichkeitseigenschaftenAbels}
 Let $R$ be a finitely generated commutative unital ring. If $R$ is \emph{infinitely generated} as a $\Z$-module, then $\phi(\mbA_n(R)) = \min\set{n-2, \phi(\BzeroR)}$  for all $n \geq 2$.
\end{conj}

For completeness we provide in Section~\ref{ExemploConjectura} a proof of Conjecture~\ref{EndlichkeitseigenschaftenAbels} in the known $S$-arithmetic cases. 

Let us briefly record a series of non-$S$-arithmetic examples of Abels' groups, again borrowing from the works of Kochloukova~\cite{Desi0} and Kropholler--Mullaney~\cite{KrophollerMullaney}.

 \begin{exm} \label{maisoutroexemplo}
Recall from Example~\ref{outroexemplo} the rings $R_\ell = \Z[x,x^{-1},(1+x)^{-1},\ldots,(\ell+x)^{-1},\frac{1}{\ell!}]$ for $\ell \in \N$. Since $\phi(\Bzero(R_\ell)) = \ell + 1$, we have that Abels' groups $\mbA_n(R_\ell)$ with $n \geq 3$ are finitely generated, and they become finitely presented if $n \geq 4$. In low dimensions one has more precisely
\[\phi(\mbA_2(R_\ell)) = 0, \, \phi(\mbA_3(R_\ell)) = 1 \, \mbox{ and } \, \phi(\mbA_4(R_\ell)) \geq 2 \, \mbox{ for all } \, \ell\in\N,\]
and the equality $\phi(\mbA_4(R_1)) = 2$ in the case $\ell = 1$.
 \end{exm}

As for the proof of Theorem~\ref{Strebel's}, our approach differs from Strebel's and both have their own advantages. His proof in~\cite{StrebelAbels} is purely algebraic and, under the needed hypotheses, he explicitly constructs a convenient finite presentation for $A_n(R,Q) = \mbU_n(R) \rtimes Q^{n-2}$. 
In contrast, the proof of Theorem~\ref{Strebel's} given here has a topological flavor. Indeed, we make use of horospherical subgroups and nerve complexes \`a la Abels--Holz~\cite{Abels0,Abels,Holz,AbelsHolz}, the early $\Sigma$-invariant for metabelian groups introduced by Bieri and Strebel himself~\cite{BieriStrebel}, and K.~S.~Brown's criterion for finite presentability~\cite{Brown0} via actions on simply-connected complexes; see Section~\ref{capdois} for details.

Although Strebel's original theorem~\cite{StrebelAbels} is slightly more general than Theorem~\ref{Strebel's} as stated, 
our proof carries over to his case as well after appropriate changes. Namely, one needs only replace the necessary conditions ``$\phi(\BzeroR) \geq 1$ (resp. $\geq 2$)'' by ``the group
\[
\left\{ \left(\begin{smallmatrix} \diamond & * \\ 0 & 1 \end{smallmatrix}\right) \in \GL_2(R) \mid * \in R, \diamond \in Q \right\}
\]
is finitely generated (resp. finitely presented).''

Further remarks about our methods and those of Strebel point to interesting phenomena concerning the structure of Abels' groups; see Section~\ref{lombra}.

\section{Preliminaries and Notation} \label{prelim}

The facts collected here are standard. The reader is referred e.g. to the classics~\cite{HahnO'Meara,SGA3.3,Steinberg} and to~\cite[Chapter~7]{GeoBook} for the results on classical groups and finiteness length, respectively, that will be used throughout. Though we derived corollaries for $S$-arithmetic groups (refer to~\cite{Margulis} for more on them), familiarity with such groups is not required for the proofs of our main results. A group commutator shall be written $[x,y] = xyx^{-1}y^{-1}$. 

Given $i,j \in \set{1,\ldots,n}$ with $i \neq j$ and $r \in R$, we denote by $\eij(r) \in \GL_n(R)$ the corresponding \emph{elementary matrix} (also called elementary transvection), i.e. $\eij(r)$ is the matrix whose diagonal entries all equal $1$ and whose only off-diagonal non-zero entry is $r$ in the position $(i,j)$.

Elementary matrices and commutators between them have the following properties, which are easily checked.
\[
\eij(r)\eij(s) = \eij(r+s), \, \, [\ekl{ij}(r),\ekl{kl}(s)^{-1}] = [\ekl{ij}(r),\ekl{kl}(s)]^{-1}, \mbox{ and}
\]
\begin{equation} \label{Ancommutators}
[\ekl{ij}(r),\ekl{kl}(s)] =
\begin{cases}
\ekl{il}(rs) & \mbox{if } j=k,\\
1 & \mbox{if } i \neq l, k \neq j.
\end{cases}
\end{equation}
In particular, each subgroup 
\[\mb{E}_{ij}(R) = \spans{\,\set{\eij(r) \mid r \in R}\,} \leq \GL_n(R)\] 
is isomorphic to the underlying additive group ${\Addi}(R) =~(R,+) \cong~\set{\left( \begin{smallmatrix} 1 & r \\ 0 & 1 \end{smallmatrix}\right) \mid r \in R}$. 

Direct matrix computations also show that 
\[
 \Diag(a_1,\ldots,a_n) \Diag(b_1,\ldots,b_n) = \Diag(a_1 b_1, \ldots, a_n b_n),
\]
where $\Diag(u_1,\ldots,u_n)$ denotes the $n \times n$ diagonal matrix whose diagonal entries are $u_1,\ldots,u_n \in R^\times$. Given $i \in \set{1,\ldots,n}$, we let $D_i(R)$ denote the subgroup $\set{\Diag(u_1,\ldots,u_n) \mid u_j = 1 \mbox{ if } j \neq i} \leq \GL_n(R)$. Write $\mbD_n(R)$ for the subgroup of $\GL_n(R)$ generated by all $\Diag(u_1, \ldots, u_n)$. One then has \[\mbD_n(R) = \prod\limits_{i=1}^n D_i(R) \cong \Mult(R)^n,\] where ${\Mult}(R) = (R^\times, \cdot) \cong \GL_1(R)$ is the group of units of the base ring $R$. The matrix group scheme $\mbD_n \cong \Mult^n$, which is defined over $\Z$, is also known as the standard (maximal) torus of $\GL_n$. 
The following relations between diagonal and elementary matrices are easily verified. 
\begin{equation} \label{SteinbergRelGLn}
\Diag(u_1,\ldots,u_n) \eij(r) \Diag(u_1,\ldots,u_n)^{-1} = \eij(u_i u_j^{-1} r).
\end{equation}

Relations similar to the above hold for all classical groups. Recall that a universal Chevalley--Demazure group scheme of type $\Phi$ is the split, semi-simple, simply-connected, affine $\Z$-group scheme $\uCD$ associated to the (reduced, irreducible) root system $\Phi$. It is a result of Chevalley's that such group schemes exist, and they are unique by Demazure's theorem; refer to~\cite{SGA3.3} for existence, uniqueness, and structure theory of Chevalley--Demazure groups. Our notation below for root subgroups of $\uCD$ closely follows that of Steinberg~\cite{Steinberg}.

For every root $\alpha \in \Phi$ and ring element $r \in R$, the group $\uCD(R)$ contains a corresponding unipotent root element $x_\alpha(r) \in \uCD(R)$---these play in $\uCD(R)$ the same role as the elementary matrices $\eij(r)$ in $\GL_n(R)$. Accordingly, one has the unipotent root subgroup \[\mf{X}_\alpha(R) = \spans{\,\set{x_\alpha(r) \mid r \in R}\,} \cong \Addi(R).\] For each $\alpha \in \Phi$ and $u \in R^\times$, one also has a semi-simple root element $h_\alpha(u) \in \uCD(R)$, and we define $\mcH_\alpha(R) = \spans{\set{h_\alpha(u) \mid u \in R^\times}} \cong \Mult(R)$, the semi-simple root subgroup associated to $\alpha$. The subgroups $\mcH_\alpha(R)$ and $\mcH_\beta(R)$ commute for all roots $\alpha, \beta \in \Phi$. The (standard) torus of $\uCD(R)$ is the abelian subgroup \[\mcH(R) = \spans{\mcH_\alpha(R) \, : \, \alpha \in \Phi} \cong \Mult(R)^{\rk(\Phi)}.\]

The unipotent root subgroups in Chevalley--Demazure groups are related via Chevalley's famous commutator formulae, which generalize the commutator relations~\eqref{Ancommutators} between elementary matrices; see e.g. \cite{Steinberg,Carter,VavilovPlotkin}. As for relations between unipotent and semi-simple root elements, Steinberg derives from Chevalley's formulae a series of equations now known as Steinberg relations; cf.~\cite[p.~23]{Steinberg}. In particular, given $h_\beta(u) \in \mc{H}_\beta(R)$ and $x_\alpha(r) \in \mf{X}_\alpha(R)$, Steinberg shows 
\begin{equation} \label{SteinbergRel} 
h_\beta(u) x_\alpha(r) h_\beta(u)^{-1} = x_\alpha(u^{(\alpha, \beta)}r),
\end{equation}
where $(\alpha,\beta) := 2\frac{\spans{\alpha,\beta}}{\spans{\beta,\beta}} \in \set{0, \pm 1, \pm 2, \pm 3}$ is the corresponding Cartan integer. 

Let $W$ be the Weyl group associated to $\Phi$. The Steinberg relations~\eqref{SteinbergRel} behave well with respect to the $W$-action on the roots $\Phi$. More precisely, let $\alpha \in \Phi \subseteq \R^{\rk(\Phi)}$ and let $r_\alpha \in W$ be the associated reflection. The group $W$ has a canonical copy (up to ordering of roots) in $\uCD$ obtained via the assignment 
\[
 r_\alpha \mapsto w_\alpha := x_\alpha(1) x_{-\alpha}(1)^{-1} x_\alpha (1).
\]
With the above notation, given arbitrary roots $\beta, \gamma \in \Phi$, one has 
\begin{align} \label{SteinbergReflection} 
 \begin{split}
 h_{r_\alpha(\gamma)}(v) x_{r_\alpha(\beta)}(s) h_{r_\alpha(\gamma)}(v)^{-1} & = w_\alpha (h_\gamma(v) x_\beta(s)^{\pm 1} h_\gamma(v)^{-1}) w_\alpha^{-1} \\
 & = x_{r_\alpha(\beta)}(v^{(\beta, \gamma)}s)^{\pm 1},
 \end{split}
 \end{align}
where the sign $\pm 1$ above does not depend on $v \in R^\times$ nor on $s \in R$.

Throughout this paper we shall make repeated use of the following well-known bounds on the finiteness length.

\begin{lem} \label{obviousboundsonphi}
Given a short exact sequence $N \into G \onto Q$, 
the following hold.
\begin{enumerate}
\item \label{obphi0} If $\phi(N)$ and $\phi(Q)$ are at least $n$, then so is $\phi(G)$.
\item \label{obphi1} If $\phi(Q) = \infty$, then $\phi(N) \leq \phi(G)$.
\item \label{obphi3} If $\phi(N) = \infty$, then $\phi(Q) = \phi(G)$. 
\item \label{obphi2} If the sequence splits (equivalently, $G = N \rtimes Q$), then $\phi(G) \leq \phi(Q)$.
\item \label{obphi4} If the sequence splits trivially, i.e. $G = N \times Q$, then $\phi(G) =~\min\set{\phi(N), \phi(Q)}$.
\end{enumerate}
\end{lem}

\begin{proof}
 Part~\eqref{obphi0} is just \cite[Theorem~4(ii)]{Ratcliffe} restated in the language of finiteness length, and~\eqref{obphi1} follows immediately from~\eqref{obphi0}. Parts~\eqref{obphi3} and~\eqref{obphi2} follow from~\cite[Theorems~4(i) and~6]{Ratcliffe}. Part~\eqref{obphi4} is an immediate consequence of~\eqref{obphi2} and~\eqref{obphi0}.
\end{proof}

We point out that the finiteness length is often seen in the literature in the language of {homotopical finiteness properties}. The concept goes back to the work of C.~T.~C. Wall in the 1950s. By definition, a group $G$ is of homotopical type $\Fn{n}$ if $\phi(G) \geq n$, and of type $\Fn{\infty}$ in case $\phi(G) = \infty$; see e.g. \cite[Chapter~7]{GeoBook} for more on this. A well-known fact that will be often used throughout this work is that a finitely generated abelian group $Q$ has $\phi(Q) = \infty$, i.e. it is of type $\Fn{\infty}$ (equivalently, of type $\Fn{n}$ for all $n \in \Nzero$). This is readily seen because the torsion-free part of $Q$, say $\Z^m$ with $m \geq 0$, admits the $m$-Torus as an Eilenberg--Maclane space $K(\Z^m,1)$.

\section{Proof of Theorem~\ref{apendice}} \label{capum}

The hypotheses of Theorem~\ref{apendice} already yield an obvious bound on the finiteness length of the given group. Indeed, in the notation of Theorem~\ref{apendice}, we have that $\phi(G) \leq \phi(\mfX(R) \rtimes \mcH(R))$ by Lemma~\doubleref{obviousboundsonphi}{obphi2} because retracts preserve homotopical finiteness properties. (A group $Q$ is called a \emph{retract} of $G$ if $G$ is a semi-direct product $G \cong N \rtimes Q$ for some $N \triangleleft G$. Equivalently, the sequence $N \into G \onto Q$ splits.) The actual work thus consists in proving that $\phi(\mfX(R) \rtimes \mcH(R))$ is no greater than the desired value, $\phi(\BzeroR)$.

We begin with the following simple observation. 

\begin{rmk}
If the group of units $\Mult(R) = (R^\times, \cdot)$ is \emph{not} finitely generated, then Theorem~\ref{apendice} holds. Indeed, in this case one has 
 \[
  0 \leq \phi(G) \leq \phi(\mfX(R) \rtimes \mcH(R)) \leq \phi(\mcH(R)) \leq \phi(\Mult(R)) = 0, \mbox{ and }
 \]
 \[
  0 \leq \phi(\BzeroR) \leq \phi(\set{\left( \begin{smallmatrix} u & 0 \\ 0 & u^{-1} \end{smallmatrix} \right) \mid u \in R^\times}) = \phi(\Mult(R)) = 0,
 \]
 because both the torus $\mcH(R)$ and $\BzeroR$ retract onto $\Mult(R)$; cf. Section~\ref{prelim}.
\end{rmk}

In view of the above, we adopt the following.\\ 
\textbf{\emph{Convention throughout this section:}} Unless stated otherwise, we {assume that the group of units} $\Mult(R) = (R^\times,\cdot)$ {of the base ring} $R$ {is finitely generated}. \\

 In what follows, we denote by $\mbB_n(R)$ the subgroup of upper triangular matrices of $\GL_n(R)$. Similarly, we define $\mbB_n^\circ(R) = \mbB_n(R) \cap \SL_n(R)$. The schemes $\mbB_n$ and $\mbB_n^\circ$ are examples of Borel subgroups of classical groups. 

We start by clearing the case where our representation $\rho : G \to \mcG(R)$ has one of the affine groups as target. 
Recall that the \emph{group of affine transformations} of the base ring $R$, denoted here by $\Aff(R)$ and sometimes by $\mathrm{AGL}_1(R)$ in the literature, is the set of affine permutations $R \to R : x \mapsto ux+r$, with $u \in R^\times$ and $r \in R$. It is a standard exercise to check that $\Aff(R)$ is isomorphic to the semi-direct product $\Addi(R) \rtimes \Mult(R)$, with $\Mult(R)$ acting on $\Addi(R)$ by multiplication. 

Identifying $x \in R$ with $\left( \begin{smallmatrix} x \\ 1 \end{smallmatrix} \right)$, we readily obtain the well-known matrix representation of $\Aff(R)$ as a subgroup of $\GL_2(R)$, namely  
\[ \Aff(R) \cong \left( \begin{smallmatrix} * & * \\ 0 & 1 \end{smallmatrix} \right) \leq \GL_2(R),\]
now acting on $R \cong \{ \left( \begin{smallmatrix} * \\ 1 \end{smallmatrix} \right) \}$ via matrix multiplication. Using the relations from Section~\ref{prelim} and because the diagonal and unipotent parts of $\left( \begin{smallmatrix} * & * \\ 0 & 1 \end{smallmatrix} \right)$ intersect at the identity matrix, the above representation of $\Aff(R)$ also yields its semi-direct product decomposition 
\[ \left( \begin{smallmatrix} * & * \\ 0 & 1 \end{smallmatrix} \right) \cong \left( \begin{smallmatrix} 1 & * \\ 0 & 1 \end{smallmatrix} \right) \rtimes \left( \begin{smallmatrix} * & 0 \\ 0 & 1 \end{smallmatrix} \right) \cong \Addi(R) \rtimes \Mult(R). \]
In certain situations, it is also convenient to consider a version of $\Aff(R)$, denoted here by $\Aff_-(R)$, whose multiplicative part $\Mult(R)$ acts on $\Addi(R)$ by reverse multiplication. That is, we consider $\Aff_-(R) \cong \Addi(R) \rtimes_{-1} \Mult(R)$ with action 
\[\Addi(R) \times \Mult(R) \ni (r,u) \mapsto u^{-1}r.\] 
The group $\Aff_-(R)$ also has an obvious matrix representation, namely
\[ \Aff_-(R) \cong \left( \begin{smallmatrix} 1 & * \\ 0 & * \end{smallmatrix} \right) \leq \GL_2(R).\]

\begin{lem} \label{newlemma}
The groups $\Aff(R)$, $\BzeroR$ and $\Aff_-(R)$ are commensurable. In particular, $\phi(\Aff(R)) = \phi(\BzeroR) = \phi(\Aff_-(R))$. Thus Theorem~\ref{apendice} holds for $\mcG = \Aff$ or $\Aff_-$.
\end{lem}

\begin{proof}
Here we are under the assumption that $\Mult(R)$ is finitely generated. If $\Mult(R)$ is itself finite the lemma follows immediately, for in this case the unipotent subgroup $\left( \begin{smallmatrix} 1 & * \\ 0 & 1 \end{smallmatrix} \right)$ has finite index in the three groups considered. Otherwise $\BzeroR$ is seen to contain a subgroup of finite index which is isomorphic to a group of the form
\[
 \set{
 \left(\begin{smallmatrix}
  u^2 & r \\
  0 & 1
 \end{smallmatrix} \right)
 \in \GL_2(R) \mid u \in S^\times, r \in R
 }
\]
for some torsion-free subgroup of units $S \leq R^\times$---recall that the action of an element $\left( \begin{smallmatrix} u & 0 \\ 0 & u^{-1} \end{smallmatrix} \right) \in \BzeroR$ on $\left( \begin{smallmatrix} 1 & r \\ 0 & 1 \end{smallmatrix} \right)$ yields $\left( \begin{smallmatrix} 1 & u^2r \\ 0 & 1 \end{smallmatrix} \right)$ by the Steinberg relations~\eqref{SteinbergRel}. Since the group above is a subgroup of finite index of $\Aff(R)$, the first claim follows. Now $\Aff(R)$ is isomorphic to $\Aff_-(R)$ by inverting the action of the diagonal. More precisely, the map 
\[ \Aff(R) \ni \left( \begin{smallmatrix} 1 & r \\ 0 & 1 \end{smallmatrix} \right) \left( \begin{smallmatrix} u & 0 \\ 0 & 1 \end{smallmatrix} \right) \mapsto \left( \begin{smallmatrix} 1 & r \\ 0 & 1 \end{smallmatrix} \right) \left( \begin{smallmatrix} 1 & 0 \\ 0 & u^{-1} \end{smallmatrix} \right) \in \Aff_-(R) \]
is an isomorphism. The equalities on the finiteness length follow from the fact that $\phi$ is a quasi-isometry invariant~\cite[Corollary~9]{AlonsoFPnQI}.
\end{proof}

The following relates the finiteness lengths of Borel subgroups of $\SL_n$ and $\GL_n$.

\begin{lem} \label{BorelGLSL}
For any $R$ (with $\Mult(R)$ not necessarily finitely generated), the Borel subgroups $\mbB_n(R) \leq \GL_n(R)$ and $\mbB_n^\circ(R) \leq~\SL_n(R)$ have the same finiteness length, which in turn is no greater than $\phi(\BzeroR)$.
\end{lem}

\begin{proof}
Though stated for arbitrary (commutative, unital) rings, the proof of the lemma is essentially Bux's proof in the $S$-arithmetic case in positive characteristic~\cite[Remark~3.6]{Bux04}. Again if $\Mult(R)$ is \emph{not} finitely generated, then $\phi(\mbB_n(R)) = \phi(\mbB_n^\circ(R)) = \phi(\BzeroR) = 0$, so that we go back to our standing assumption that $\Mult(R)$ is finitely generated. 

Consider the central subgroup $\mb{Z}_n(R) \leq \mbB_n(R)$ given by 
\[
 \mb{Z}_n(R) = \set{\Diag(u,\ldots,u) \mid u \in  R^\times} = \set{u\mb{1}_n \mid u \in R^\times} \cong \Mult(R).
\]
Its intersection with $\BzeroR \leq \SL_n(R)$ is the group $\mu_n(R)$ of $n$-th roots of unity of $R$, i.e. 
\[
\mu_n(R) \cong \mb{Z}_n(R) \cap \mbB_n^\circ(R) = \set{u\mb{1}_n \mid u \in R^\times \mbox{ and } u^n = 1}. 
\]
Given an arbitrary abelian group $A$ and $m \in \N$, let us denote by $(A)^m \leq A$ the subgroup of $m$-th powers, that is $(A)^m = \{ a^m \mid a \in A \}$. We observe that the determinant map, when restricted to $\mb{Z}_n(R) \leq \mbB_n(R)$, induces the $n$-th power map $u \mapsto u^n$ on $\mb{Z}_n(R) \cong \Mult(R)$, the kernel of this map being precisely $\mu_n(R)$. 

Thus, using the determinant and passing over to the projective groups $\PmbB_n^\circ(R) := \mbB_n^\circ(R)/\mu_n(R)$ and $\PmbB_n(R) := \mbB_n(R) / \mb{Z}_n(R)$, we obtain the following commutative diagram of short exact sequences. 
\begin{center}
\begin{tikzpicture}
 \node (Bn) {$\mbB_n(R)$};
 \node (Bn0) [left=of Bn] {$\mbB_n^\circ(R)$};
 \node (R*) [right=of Bn] {$\Mult(R)$};
 \node (Z0) [above=of Bn0] {$\mu_n(R)$};
 \node (Z) [above=of Bn] {$\mb{Z}_n(R)$};
 \node (R*n) [above=of R*] {$(\Mult(R))^n$};
 \node (PBn0) [below=of Bn0] {$\PmbB_n^\circ(R)$};
 \node (PBn) [below=of Bn] {$\PmbB_n(R)$};
 \node (finite) [below=of R*] {$\frac{\Mult(R)}{(\Mult(R))^n}$};
 \draw[->>] (Bn) to node [auto] {det} (R*);
 \draw[->>] (Bn) to (PBn);
 \draw[->>] (R*) to (finite);
 \draw[->>] (PBn) to (finite);
 \draw[->>] (Z) to node [auto] {\tiny{$u \mapsto u^n$}} (R*n);
 \draw[->>] (Bn0) to (PBn0);
 \draw[arrows = {Hooks[right]->}] (R*n) to (R*);
 \draw[arrows = {Hooks[right]->}] (PBn0) to (PBn);
 \draw[arrows = {Hooks[right]->}] (Z0) to (Z);
 \draw[arrows = {Hooks[right]->}] (Z0) to (Bn0);
 \draw[arrows = {Hooks[right]->}] (Z) to (Bn);
 \draw[arrows = {Hooks[right]->}] (Bn0) to (Bn);
\end{tikzpicture}
\end{center}
Since $\Mult(R)$ is finitely generated abelian, the groups $\mu_n(R)$ and ${\Mult(R)}/{(\Mult(R))^n}$ are finite, from which 
$\phi(\mbB_n^\circ(R)) = \phi(\PmbB_n^\circ(R))$ { and } $\phi(\PmbB_n^\circ(R)) = \phi(\PmbB_n(R))$ 
follow by Lemma~\ref{obviousboundsonphi}\eqref{obphi3} and quasi-isometry invariance of $\phi$ \cite[Corollary~9]{AlonsoFPnQI}. But again Lemma~\ref{obviousboundsonphi}\eqref{obphi3} yields $\phi(\mbB_n(R)) = \phi(\PmbB_n(R))$, whence the first claim of the lemma. 

Finally, any $\mbB_n(R)$ retracts onto $\mbB_2(R)$ via the map
\begin{center}
 \begin{tikzpicture}
  \node (Bn) {$
   \mbB_n(R) = 
 \left(\begin{smallmatrix}
  * & * & * & \cdots & * \\
  0 & * & * & \ddots & \vdots \\
  0 & 0 & * & \ddots & \vdots \\
  \vdots & \ddots & \ddots & \ddots & * \\
  0 & \cdots & \cdots & 0 & *
 \end{smallmatrix}\right)
  $};
  \node (B2) [right=of Bn] {$
    \left(\begin{smallmatrix}
  * & * & 0 & \cdots & 0 \\
  0 & * & 0 & \ddots & \vdots \\
  0 & 0 & 1 & \ddots & \vdots \\
  \vdots & \ddots & \ddots & \ddots & 0 \\
  0 & \cdots & \cdots & 0 & 1
 \end{smallmatrix}\right)
 \cong \mbB_2(R)$,
  };
  \draw[->>] (Bn) to (B2);
 \end{tikzpicture}
\end{center}
which yields the second claim.
\end{proof}

To prove the desired inequality $\phi(\mfX(R) \rtimes \mcH(R)) \leq \phi(\BzeroR)$, we shall use well-known matrix representations of classical groups. 
We warm-up by settling the simpler case where the given classical group $\mcG$ containing $\mfX \rtimes \mcH$ is the general linear group itself, which will set the tune for the remaining cases. (Recall that $\mcH$ is a maximal torus of $\mcG$.)

\begin{pps} \label{RTGLn}
Theorem~\ref{apendice} holds if $\mcG = \GL_n$.
\end{pps}

\begin{proof} 
Here we take a matrix representation of $\mcG = \GL_n$ such that the given soluble subgroup $\mfX \rtimes \mcH$ is upper triangular. In this case, the maximal torus $\mcH$ is the subgroup of diagonal matrices of $\GL_n$, i.e. 
\[
\mcH(R) \cong \mbD_n(R) = \prod\limits_{i=1}^n D_i(R) \leq \GL_n(R),
\]
and $\mfX$ is identified with a subgroup of elementary matrices in a single fixed position, say $(i,j)$ with $i < j$. That is, 
\[
\mfX(R) \cong \mb{E}_{ij}(R) = \spans{\set{\eij(r) \mid r \in R}} \leq \GL_n(R).
\]
Recall that the action of the torus $\mcH(R) \cong \mbD_n(R)$ on the unipotent root subgroup $\mfX(R) \cong \mb{E}_{ij}(R)$ is given by the diagonal relations~\eqref{SteinbergRelGLn}. But such relations also imply the decomposition
\begin{align*}
\mb{E}_{ij}(R) \rtimes \mbD_n(R) &= \spans{\mb{E}_{ij}(R), D_i(R), D_j(R)} \times \prod\limits_{i \neq k \neq j} D_k(R) \\
&\cong \mbB_2(R) \times \Mult(R)^{n-2} 
\end{align*}
because all diagonal subgroups $D_k(R)$ with $k \neq i,j$ act trivially on the elementary matrices $\eij(r)$. Since we are assuming $\Mult(R)$ to be finitely generated, it follows from Lemmata~\doubleref{obviousboundsonphi}{obphi4} and~\ref{BorelGLSL} that
\[
\phi(\mfX(R) \rtimes \mcH(R)) = \min\set{\phi(\mbB_2(R)), \phi(\Mult(R)^{n-2})} = \phi(\BzeroR).
\]
\end{proof}

It remains to investigate the situation where the classical group $\mcG$ in the statement of Theorem~\ref{apendice} is a universal Chevalley--Demazure group scheme. Write $\mcG = \uCD$, with underlying root system $\Phi$ associated to the given maximal torus $\mcH \leq \uCD$ and with a fixed set of simple roots ${\Delta} \subset \Phi$. One has 
\[
 \mcH(R) = \prod\limits_{\alpha \in \Delta} \mcH_\alpha(R),
\]
and $\mfX$ is the unipotent root subgroup associated to some (positive) root $\eta \in \Phi^+$, that is,
\[
 \mfX(R) = \mfX_\eta(R) = \spans{\,\set{x_\eta(r) \mid r \in R}\,}.
\]
The proof proceeds by a case-by-case analysis on $\Phi$ and $\eta$. Instead of diving straight into all possible cases, some obvious reductions can be done.

\begin{lem} \label{obviousreductionforRT}
If Theorem~\ref{apendice} holds when $\mcG$ is a universal Chevalley--Demazure group scheme $\uCD$ of rank at most four and $\mfX = \mfX_\eta$ with $\eta \in \Phi^+$ simple, then it holds for any universal Chevalley--Demazure group scheme.
\end{lem}

\begin{proof}
 Write $\mfX(R) = \mfX_\eta(R)$ and $\mcH(R) = \prod_{\alpha \in \Delta} \mcH_\alpha(R)$ as above. By~\eqref{SteinbergReflection} we can find an element $w$ in the Weyl group $W$ of $\Phi$ and a corresponding element $\omega \in \uCD(R)$ such that $w(\eta) \in \Phi^+$ is a simple root and 
 \[
  \omega (\mfX_\eta(R) \rtimes \mcH(R)) \omega^{-1} \cong \mfX_{w(\eta)}(R) \rtimes \mcH(R).
 \]
(The conjugation aboves takes place in the overgroup $\uCD(R)$.) We may thus assume $\eta \in \Phi^+$ to be simple. From the Steinberg relations~\eqref{SteinbergRel} we have that
 \begin{align*}
  \mfX(R) \rtimes \mcH(R) = \left( \mfX_\eta(R) \rtimes \left( \underset{\spans{\eta, \alpha} \neq 0}{\prod\limits_{\alpha \in \Delta}} \mcH_\alpha(R) \right) \right) \times \underset{\spans{\eta, \beta} = 0}{\prod\limits_{\beta \in \Delta}} \mcH_\beta(R),
 \end{align*}
 yielding $\phi(\mfX(R) \rtimes \mcH(R)) = \phi(\mfX_\eta(R) \rtimes \mcH^\circ(R))$ by Lemma~\doubleref{obviousboundsonphi}{obphi4}, where 
 \[
  \mcH^\circ(R) = \underset{\spans{\eta, \alpha} \neq 0}{\prod\limits_{\alpha \in \Delta}} \mcH_\alpha(R).
 \]
 Inspecting all possible Dynkin diagrams, it follows that the number of simple roots $\alpha \in \Delta$ for which $\spans{\eta, \alpha} \neq 0$ is at most four. The lemma follows.
\end{proof}

Thus, in view of Proposition~\ref{RTGLn} and Lemma~\ref{obviousreductionforRT}, the proof of Theorem~\ref{apendice} will be complete once we establish the following.

\begin{pps} \label{RTcasebycase}
 Theorem~\ref{apendice} holds whenever $\mcG$ is a universal Chevalley--Demazure group scheme $\uCD$ with 
 \[
\Phi \in \set{\tA_1, \tA_2, \tC_2, \tG, \tA_3, \tB_3, \tC_3, \tD_4}  
 \]
 and $\mfX \rtimes \mcH$ of the form 
 \[
\mfX \rtimes \mcH = \mfX_\eta \rtimes \left( \underset{\spans{\eta, \alpha} \neq 0}{\prod\limits_{\alpha \in \Delta}} \mcH_\alpha \right) \mbox{ with } \eta \in \Phi^+ \mbox{ simple.}
 \]
\end{pps}

\begin{proof}
 The idea of the proof is quite simple. In each case, we find a matrix group $G(\Phi,\eta,R)$ satisfying $\phi(G(\Phi,\eta,R)) = \phi(\BzeroR)$ and which fits into a short exact sequence
 \[
  \mfX_\eta(R) \rtimes \mcH(R) \into G(\Phi,\eta,R) \onto Q(\Phi,\eta,R)
 \]
where $Q(\Phi,\eta,R)$ is finitely generated abelian. In fact, $G(\Phi,\eta,R)$ can often be taken to be $\mfX_\eta(R) \rtimes \mcH(R)$ itself so that $Q(\Phi,\eta,R)$ is trivial in many cases. The proposition then follows from Lemma~\doubleref{obviousboundsonphi}{obphi1}.

To construct the matrix groups $G(\Phi,\eta,R)$ above, we use mostly Ree's matrix representations of classical groups~\cite{Ree} as worked out by Carter in~\cite{Carter}. (Recall that the case of Type $\tB_n$ was cleared by Dieudonn\'e~\cite{Dieudonne} after left open in Ree's paper.) In the exceptional case $\tG$ we follow Seligman's identification from~\cite{SeligmanLieAlgII}. We remark that Seligman's numbering of indices agrees with that of Carter's for $\text{G}_2$ as a subalgebra of $\text{B}_3$.

\underline{Type A:} Identify $\mcG^{{\rm sc}}_{A_n}$ with $\SL_{n+1}$ so that the soluble subgroup $\mfX_\eta \rtimes \mcH \leq \SL_{n+1}$ is upper triangular and the given maximal torus $\mcH$ of $\SL_{n+1}$ is the subgroup of diagonal matrices. Now, if $\rk(\Phi)=1$, then there is nothing to check, for in this case $\mfX_\eta(R) \rtimes \mcH(R)$ itself is isomorphic to $\BzeroR$. If $\Phi = \tA_2$, identify $\mfX_\eta(R)$ with the root subgroup $\mb{E}_{12}(R) \leq \SL_3(R)$ so that 
\[
 \mfX_\eta(R) \rtimes \mcH(R) \cong \set{ 
 \left(\begin{smallmatrix}
   a & r & 0 \\ 0 & b & 0 \\ 0 & 0 & (ab)^{-1}
  \end{smallmatrix} \right) 
  \in \SL_3(R)
  \, \middle| \, a, b \in R^\times, r \in R }.
\]
In this case, $G(\tA_2,\eta,R) := \mf{X}_\eta(R) \rtimes \mcH(R)$ is isomorphic to $\mbB_2(R)$ via
\begin{center}
 \begin{tikzpicture}
  \node (XH) {
  $ \mfX_\eta(R) \rtimes \mcH(R) \ni
  \left(\begin{smallmatrix}
   a & r & 0 \\ 0 & b & 0 \\ 0 & 0 & (ab)^{-1}
  \end{smallmatrix} \right)$
  };
  \node (B2) [right=of XH] {
  $\begin{pmatrix}
   a & r \\ 0 & b 
  \end{pmatrix} \in \mbB_2(R) \leq \GL_2(R)$.
  };
  \draw[arrows = {Bar[]->}] (XH) to (B2);
 \end{tikzpicture}
\end{center}
The case $\tA_2$ thus follows from Lemma~\ref{BorelGLSL}. 
If now $\Phi = \tA_3$, we identify $\mfX_\eta(R)$ with the root subgroup $\mb{E}_{23}(R) \leq \SL_4(R)$, which gives 
\[
 \mfX_\eta(R) \rtimes \mcH(R) \cong \set{ 
 \left(\begin{smallmatrix}
   a & 0 & 0 & 0 \\ 0 & b & r & 0 \\ 0 & 0 & c & 0 \\ 0 & 0 & 0 & (abc)^{-1}
  \end{smallmatrix} \right) 
  \in \SL_3(R)
  \, \middle| \, a, b, c \in R^\times, r \in R }.
\]
Here, $G(\tA_3,\eta,R) := \mfX_\eta(R) \rtimes \mcH(R)$ is isomorphic to the group $\mbB_2(R) \times~\Mult(R)$ via the map
\begin{center}
 \begin{tikzpicture}
  \node (XH) {
  $ 
  \left(\begin{smallmatrix}
   a & 0 & 0 & 0 \\ 0 & b & r & 0 \\ 0 & 0 & c & 0 \\ 0 & 0 & 0 & (abc)^{-1}
  \end{smallmatrix} \right)$
  };
  \node (B2) [right=of XH] {
  $\left( 
  \begin{pmatrix}
   b & r \\ 0 & c 
  \end{pmatrix}, a
  \right)$.
  };
  \draw[arrows = {Bar[]->}] (XH) to (B2);
 \end{tikzpicture}
\end{center}
Thus, $\phi(\mfX_\eta(R) \rtimes \mcH(R)) = \phi(\BzeroR)$ by Lemmata~\doubleref{obviousboundsonphi}{obphi4} and~\ref{BorelGLSL}. 

\underline{Type C:} Suppose $\Phi = \text{C}_n$. Following Ree and Carter we identify $\mcG^{{\rm sc}}_{\text{C}_n}$ with the symplectic group ${\rm Sp}_{2n} \leq \SL_{2n}$. If $\Phi = \text{C}_2$, denote by $\Delta = \set{\alpha, \beta}$ the set of simple roots, where $\alpha$ is short and $\beta$ is long. The unipotent root subgroups are given by
\[
 \mfX_\alpha(R) \cong \spans{\set{\ekl{12}(r)\ekl{43}(r)^{-1} \in \SL_4(R) \mid r \in R \}}} \, \mbox{ and}
\]
\[
 \mfX_\beta(R) \cong \mb{E}_{24}(R) = \spans{\set{\ekl{24}(r) \in \SL_4(R) \mid r \in R \}}},
\]
whereas the maximal torus $\mcH(R)$ is the diagonal subgroup
\begin{align*}
 \mcH(R) = \spans{\mcH_\alpha(R), \mcH_\beta(R)} \cong \langle \{ & \Diag(a,a^{-1},a^{-1},a), \\
   & \Diag(1,b,1,b^{-1}) \in \SL_4(R) \mid a,b \in R^\times \} \rangle.
\end{align*}
Now, if $\eta = \alpha$ (that is, if $\eta$ is short), then $\mfX_\eta(R) \rtimes \mcH(R)$ is the group
\[
\mfX_\eta(R) \rtimes \mcH(R) \cong \set{
\left( \begin{smallmatrix}
        u & r & 0 & 0 \\
        0 & v & 0 & 0 \\
        0 & 0 & u^{-1} & 0 \\
        0 & 0 & -r & v^{-1}
       \end{smallmatrix} \right)
       \in \Sp_4(R) \, \middle| \, u, v \in R^\times, r \in R
}. 
\]
Hence, $G(\tC_2,\eta,R) := \mfX_\eta(R) \rtimes \mcH(R)$ is isomorphic to $\mbB_2(R)$ via
\begin{center}
 \begin{tikzpicture}
  \node (XH) {$ 
  \mfX_\eta(R) \rtimes \mcH(R) \ni
  \left( \begin{smallmatrix}
        u & r & 0 & 0 \\
        0 & v & 0 & 0 \\
        0 & 0 & u^{-1} & 0 \\
        0 & 0 & -r & v^{-1}
       \end{smallmatrix} \right)
       $};
  \node (B2) [right=of XH] {$
  \begin{pmatrix}
   u & r \\ 
   0 & v 
  \end{pmatrix}
  \in \mbB_2(R)
  $,};
  \draw[arrows = {Bar[]->}] (XH) to (B2);
 \end{tikzpicture}
\end{center}
which yields $\phi(\mfX_\eta(R) \rtimes \mcH(R)) = \phi(\BzeroR)$ by Lemma~\ref{BorelGLSL}. On the other hand, if $\eta = \beta$ (i.e. $\eta$ is long), then $\mfX_\eta(R) \rtimes \mcH(R)$ is given by
\[
 \mfX_\eta(R) \rtimes \mcH(R) \cong \set{
\left( \begin{smallmatrix}
        u & 0 & 0 & 0 \\
        0 & v & 0 & r \\
        0 & 0 & u^{-1} & 0 \\
        0 & 0 & 0 & v^{-1}
       \end{smallmatrix} \right)
       \in \Sp_4(R) \, \middle| \, u, v \in R^\times, r \in R
}, 
\]
which is isomorphic to $\BzeroR \times \Mult(R)$ via
\begin{center}
 \begin{tikzpicture}
  \node (XH) {$ 
  \mfX_\eta(R) \rtimes \mcH(R) \ni
  \left( \begin{smallmatrix}
        u & 0 & 0 & 0 \\
        0 & v & 0 & r \\
        0 & 0 & u^{-1} & 0 \\
        0 & 0 & 0 & v^{-1}
       \end{smallmatrix} \right)
       $};
  \node (B02R*) [right=of XH] {$
  \left( \left(\begin{smallmatrix}
   v & r \\ 
   0 & v^{-1} 
  \end{smallmatrix}\right),
  u \right)
  \in \BzeroR \times \Mult(R)
  $.};
  \draw[arrows = {Bar[]->}] (XH) to (B02R*);
 \end{tikzpicture}
\end{center}
Thus, $\phi(\mfX_\eta(R) \rtimes \mcH(R)) = \phi(\BzeroR)$ by Lemma~\doubleref{obviousboundsonphi}{obphi4}. 

Lastly, assume $\Phi = \text{C}_3$ and denote its set of simple roots by $\Delta=\set{\alpha_1, \alpha_2, \beta}$, where $\beta$ is the long root. We have $\mcG^{{\rm sc}}_{\text{C}_3} = \Sp_6$ with the root subgroups given by the following matrix subgroups.
\[
 \mfX_{\alpha_1}(R) \cong \spans{\set{\ekl{12}(r)\ekl{54}(r)^{-1} \in \SL_6(R) \mid r \in R}},
\]
\[
 \mfX_{\alpha_2}(R) \cong \spans{\set{\ekl{23}(r)\ekl{65}(r)^{-1} \in \SL_6(R) \mid r \in R}},
\]
\[
 \mfX_{\beta}(R) \cong \mb{E}_{36}(R) = \spans{\set{\ekl{36}(r) \in \SL_6(R) \mid r \in R}},
\]
and
\begin{align*}
 & \mcH(R) = \spans{\mcH_{\alpha_1}(R), \mcH_{\alpha_2}(R), \mcH_\beta(R)} \cong \langle \, \{ \Diag(a_1, a_1^{-1}, 1, a_1^{-1}, a_1, 1), \\
 & \Diag(1,a_2,a_2^{-1},1,a_2^{-1},a_2), \Diag(1, 1, b, 1, 1, b^{-1}) \mid a_1, a_2, b \in R^\times \}\, \rangle.
\end{align*}
Here we are only interested in the case where $\eta$ is the central root $\alpha_2$, for otherwise $\eta$ would be orthogonal to one of the other simple roots.
Thus,
\[
 \mfX_\eta(R) \rtimes \mcH(R) \cong \set{
\left( \begin{smallmatrix}
        u & 0 & 0 & 0 & 0 & 0 \\
        0 & v & r & 0 & 0 & 0 \\
        0 & 0 & w & 0 & 0 & 0 \\
        0 & 0 & 0 & u^{-1} & 0 & 0 \\
        0 & 0 & 0 & 0 & v^{-1} & 0 \\
        0 & 0 & 0 & 0 & -r & w^{-1}
       \end{smallmatrix} \right)
       \in \Sp_6(R) \, \middle| \, u, v, w \in R^\times, r \in R
}. 
\]
The isomorphism
\begin{center}
 \begin{tikzpicture}
  \node (XH) {$ 
\left( \begin{smallmatrix}
        u & 0 & 0 & 0 & 0 & 0 \\
        0 & v & r & 0 & 0 & 0 \\
        0 & 0 & w & 0 & 0 & 0 \\
        0 & 0 & 0 & u^{-1} & 0 & 0 \\
        0 & 0 & 0 & 0 & v^{-1} & 0 \\
        0 & 0 & 0 & 0 & -r & w^{-1}
       \end{smallmatrix} \right)
       $};
  \node (B2R*) [right=of XH] {$
  \left( \begin{pmatrix}
   v & r \\ 
   0 & w 
  \end{pmatrix},
  u \right)
  $};
  \draw[arrows = {Bar[]->}] (XH) to (B2R*);
 \end{tikzpicture}
\end{center}
between $G(\tC_3,\eta,R) := \mfX_\eta(R) \rtimes \mcH(R)$ and $\mbB_2(R) \times \Mult(R)$ 
then yields $\phi(\mfX_\eta(R) \rtimes \mcH(R)) = \phi(\BzeroR)$ by Lemmata~\doubleref{obviousboundsonphi}{obphi4} and~\ref{BorelGLSL}.

\underline{Type D:} The case of maximal rank concerns the root system $\Phi = \text{D}_4$, with set of simple roots $\set{\alpha_1, \alpha_2, \alpha_3, \alpha_4}$ and the given simple root $\eta$ being equal to the central root $\alpha_2$ which is \emph{not} orthogonal to any other simple root. Here, $\mcG^{{\rm sc}}_{\text{D}_4} = \mathrm{Spin}_8$. Following Ree and Carter, the root subgroups and the maximal torus are given as follows.
\[
 \mfX_{\alpha_1}(R) \cong \spans{\set{\ekl{12}(r)\ekl{65}^{-1} \in \SL_8(R) \mid r \in R}},
\]
\[
 \mfX_{\alpha_2}(R) \cong \spans{\set{\ekl{23}(r)\ekl{76}^{-1} \in \SL_8(R) \mid r \in R}},
\]
\[
 \mfX_{\alpha_3}(R) \cong \spans{\set{\ekl{34}(r)\ekl{87}^{-1} \in \SL_8(R) \mid r \in R}},
\]
\[
 \mfX_{\alpha_4}(R) \cong \spans{\set{\ekl{38}(r)\ekl{47}^{-1} \in \SL_8(R) \mid r \in R}},
\]
and
\begin{align*}
 \mcH(R) = & \spans{\mcH_{\alpha_1}(R), \, \mcH_{\alpha_2}(R), \, \mcH_{\alpha_3}(R), \,  \mcH_{\alpha_4}(R)} \\ \cong & \langle \, \{ \Diag(a_1, a_1^{-1}, 1, 1, a_1^{-1}, a_1, 1, 1), \, \, \Diag(1, a_2, a_2^{-1}, 1, 1, a_2^{-1}, a_2, 1), \\
 & \Diag(1, 1, a_3, a_3^{-1}, 1, 1, a_3^{-1}, a_3), \quad \Diag(1, 1, a_4, a_4, 1, 1, a_4^{-1}, a_4^{-1}) \\
 & \in \SL_8(R) \mid a_1, a_2, a_3, a_4 \in R^\times \} \, \rangle.
\end{align*}
The torus $\mcH(R)$ is a subgroup of the following diagonal group.
\[
 T(R) = \spans{ \, \set{\Diag(u,v,w,x,u^{-1},v^{-1},w^{-1},x^{-1}) \mid u,v,w,x \in R^\times}\,} \leq  \SL_8(R).
\]
Recall that $\eta = \alpha_2$, the central root. Set $G(\tD_4,\alpha_2,R) := \mfX_\eta(R) \rtimes T(R)$. We have a short exact sequence
\[
 \mfX_\eta(R) \rtimes \mcH(R) \into G(\tD_4,\alpha_2,R) \onto \frac{T(R)}{\mcH(R)},
\]
with quotient $T(R)/\mcH(R)$ finitely generated abelian. But $G(\tD_4,\alpha_2,R)$ is isomorphic to $\mbB_2(R) \times \Mult(R)^2$ via
\begin{center}
 \begin{tikzpicture}
  \node (XH) {$ 
  G(\text{D}_4,\alpha_2,R) \ni
\left( \begin{smallmatrix}
        u & 0 & 0 & 0 & 0 & 0 & 0 & 0 \\
        0 & v & r & 0 & 0 & 0 & 0 & 0 \\
        0 & 0 & w & 0 & 0 & 0 & 0 & 0 \\
        0 & 0 & 0 & x & 0 & 0 & 0 & 0 \\
        0 & 0 & 0 & 0 & u^{-1} & 0 & 0 & 0 \\
        0 & 0 & 0 & 0 & 0 & v^{-1} & 0 & 0 \\
        0 & 0 & 0 & 0 & 0 & -r & w^{-1} & 0 \\
        0 & 0 & 0 & 0 & 0 & 0 & 0 & x^{-1}
       \end{smallmatrix} \right)
       $};
  \node (B2R*2) [right=of XH] {$
  \left( \left( \begin{smallmatrix}
   v & r \\ 
   0 & w 
  \end{smallmatrix} \right),
  u, x \right)
  $,};
  \draw[arrows = {Bar[]->}] (XH) to (B2R*2);
 \end{tikzpicture}
\end{center}
whence $\phi(\mf{X}_\eta(R) \rtimes \mcH(R)) \leq \phi(\BzeroR)$ by Lemmata~\doubleref{obviousboundsonphi}{obphi4} and~\ref{BorelGLSL}.

\underline{Types B and G:} We approach the remaining cases using the embedding of the group of type $\tG$ into the spin group of type $\tB_3$. Assume for the remainder of the proof that the base ring $R$ has $\carac(R) \neq 2$ in order to simplify the choice of a symmetric matrix preserved by the elements of $\mcG^{{\rm sc}}_{\tB_3}(R) = \mathrm{Spin}_7(R)$. This assumption is harmless, for the proof in the case $\carac(R) = 2$ follows analogously (after a change of basis) using Dieudounn\'e's matrix representation~\cite{Dieudonne}.

Denote by $\Delta = \set{\alpha_1, \alpha_2, \beta}$ the set of simple roots of $\tB_3$, where $\beta$ is the short root. The root subgroups are given below.
\[
 \mfX_{\alpha_1}(R) \cong \spans{\set{\ekl{23}(r)\ekl{65}(r)^{-1} \in \SL_7(R) \mid r \in R}},
\]
\[
 \mfX_{\alpha_2}(R) \cong \spans{\set{\ekl{34}(r)\ekl{76}(r)^{-1} \in \SL_7(R) \mid r \in R}},
\]
\[
 \mfX_{\beta}(R) \cong \spans{\set{\exp(r \cdot (2 \cdot \Ekl{41} - \Ekl{17})) \in \SL_7(R) \mid r \in R}},
\]
\[
 \mcH_{\alpha_1}(R) \cong \spans{\set{\Diag(1,a_1,a_1^{-1},1,a_1^{-1},a_1,1) \in \SL_7(R) \mid a_1 \in R^\times}},
\]
\[
 \mcH_{\alpha_2}(R) \cong \spans{\set{\Diag(1,1,a_2,a_2^{-1},1,a_2^{-1},a_2) \in \SL_7(R) \mid a_2 \in R^\times}},
\]
and
\[
 \mcH_{\beta}(R) \cong \spans{\set{\Diag(1,1,1,b^2,1,1,b^{-2}) \in \SL_7(R) \mid b \in R^\times}}.
\]
Now let $\Lambda = \set{\alpha, \gamma}$ denote the set of simple roots of $\tG$, where $\gamma$ is the short root. In the identification above, the embedding of $\tG$ into $\tB_3$ maps the long root $\alpha \in \tG$ to the (long) root $\alpha_1 \in \tB_3$, and the root subgroups of $\mcG^{{\rm sc}}_{\tG} \leq \mcG^{{\rm sc}}_{\tB_3}$ are listed below.
\[
 \mfX_\alpha(R) = \mfX_{\alpha_1}(R),
\]
\[
 \mfX_{\gamma}(R) \cong \spans{\set{\exp(r \cdot (2 \cdot \Ekl{12} + \Ekl{37} - \Ekl{46} - \Ekl{51})) \in \SL_7(R) \mid r \in R}},
\]
\[
 \mcH_{\alpha}(R) = \mcH_{\alpha_1}(R),
\]
and
\[
 \mcH_{\gamma}(R) \cong \spans{\set{\Diag(1,c^{-2},c,c,c^2,c^{-1},c^{-1}) \in \SL_7(R) \mid c \in R^\times}}.
\]

We now return to the soluble subgroup $\mfX_\eta \rtimes \mcH \leq \uCD$. In the case $\Phi = \tB_3$, the maximal torus $\mcH(R)$ is the diagonal subgroup $\mcH(R) = \spans{\mcH_{\alpha_1}(R), \, \mcH_{\alpha_2}(R), \, \mcH_\beta(R)}$ and $\eta$ is the middle simple root $\alpha_2$ which is not orthogonal to the other simple roots, so that $\mfX_\eta(R) = \mfX_{\alpha_2}(R)$. Let $T(R)$ be the diagonal group
\[
 T(R) = \spans{\set{\Diag(1,u,v,w,u^{-1},v^{-1},w^{-1}) \in \SL_7(R) \mid u,v,w \in R^\times}}.
\]
Setting $G(\tB_3,\alpha_2,R) := \mfX_\eta(R) \rtimes T(R)$ we obtain a short exact sequence
\[
 \mfX_\eta(R) \rtimes \mcH(R) \into G(\tB_3,\alpha_2,R) \onto \frac{T(R)}{\mcH(R)},
\]
which gives $\phi(\mfX_\eta(R) \rtimes \mcH(R)) \leq \phi(G(\tB_3,\alpha_2,R))$ by Lemma~\doubleref{obviousboundsonphi}{obphi1}. But the isomorphism $G(\tB_3,\alpha_2,R) \cong \mbB_2(R) \times \Mult(R)$ given by
\begin{center}
 \begin{tikzpicture}
  \node (XH) {$ 
  G(\text{B}_3,\alpha_2,R) \ni
\left( \begin{smallmatrix}
        1 & 0 & 0 & 0 & 0 & 0 & 0 \\
        0 & u & 0 & 0 & 0 & 0 & 0 \\
        0 & 0 & v & r & 0 & 0 & 0 \\
        0 & 0 & 0 & w & 0 & 0 & 0 \\
        0 & 0 & 0 & 0 & u^{-1} & 0 & 0 \\
        0 & 0 & 0 & 0 & 0 & v^{-1} & 0 \\
        0 & 0 & 0 & 0 & 0 & -r & w^{-1}
       \end{smallmatrix} \right)
       $};
  \node (B2R*) [right=of XH] {$
  \left( \left( \begin{smallmatrix}
   v & r \\ 
   0 & w 
  \end{smallmatrix} \right),
  u \right)
  $,};
  \draw[arrows = {Bar[]->}] (XH) to (B2R*);
 \end{tikzpicture}
\end{center}
yields $\phi(G(\tB_3,\alpha_2,R)) = \phi(\BzeroR)$ by Lemmata~\doubleref{obviousboundsonphi}{obphi4} and~\ref{BorelGLSL}.

Suppose now that $\Phi = \tG$. The maximal torus $\mcH(R)$ is the diagonal subgroup $\mcH(R) = \spans{\mcH_{\alpha}(R), \, \mcH_\gamma(R)}$. This time we consider the diagonal subgroup 
\[
 T(R) = \spans{\set{\Diag(1,u,v,u^{-1}v^{-1},u^{-1},v^{-1},uv) \in \SL_7(R) \mid u,v \in R^\times}}
\]
and let $G(\tG,\eta,R) := \mfX_\eta(R) \rtimes T(R)$, again obtaining a short exact sequence $\mfX_\eta(R) \rtimes \mcH(R) \into G(\text{G}_2,\eta,R) \onto T(R)/\mcH(R)$ which yields $\phi(\mfX_\eta(R) \rtimes~\mcH(R)) \leq \phi(G(\tG,\eta,R))$. If $\eta$ is the long root $\alpha = \alpha_1$, then the map
\begin{center}
 \begin{tikzpicture}
  \node (XH) {$ 
  G(\tG,\eta,R) \ni
\left( \begin{smallmatrix}
        1 & 0 & 0 & 0 & 0 & 0 & 0 \\
        0 & u & r & 0 & 0 & 0 & 0 \\
        0 & 0 & v & 0 & 0 & 0 & 0 \\
        0 & 0 & 0 & u^{-1}v^{-1} & 0 & 0 & 0 \\
        0 & 0 & 0 & 0 & u^{-1} & 0 & 0 \\
        0 & 0 & 0 & 0 & -r & v^{-1} & 0 \\
        0 & 0 & 0 & 0 & 0 & 0 & uv
       \end{smallmatrix} \right)
       $};
  \node (B2) [right=of XH] {$
  \begin{pmatrix}
   u & r \\ 
   0 & v 
  \end{pmatrix}
  $};
  \draw[arrows = {Bar[]->}] (XH) to (B2);
 \end{tikzpicture}
\end{center}
yields an isomorphism $G(\tG,\eta,R) \cong \mbB_2(R)$, whence $\phi(G(\tG,\eta,R)) = \phi(\BzeroR)$ by Lemma~\ref{BorelGLSL}. If $\eta$ is the short root $\gamma$, we observe that
\[
 G(\tG,\eta,R) = \mfX_\gamma(R) \rtimes \mcH(R) \cong (\mfX_\gamma(R) \rtimes \mcH_{\alpha}(R)) \times \Mult(R)
\]
because the following holds for any $x_\gamma(r) = \exp(r \cdot (2 \cdot \Ekl{12} + \Ekl{37} - \Ekl{46} - \Ekl{51})) \in \mfX_\gamma(R)$ and $d~=~\Diag(1,u,v,u^{-1}v^{-1},u^{-1},v^{-1},uv) \in T(R)$.
\[
 d x_\gamma(r) d^{-1} =
 d 
 \left( \begin{smallmatrix}
        1 & 2r & 0 & 0 & 0 & 0 & 0 \\
        0 & 1 & 0 & 0 & 0 & 0 & 0 \\
        0 & 0 & 1 & 0 & 0 & 0 & r \\
        0 & 0 & 0 & 1 & 0 & -r & 0 \\
        -r & -r^2 & 0 & 0 & 1 & 0 & 0 \\
        0 & 0 & 0 & 0 & 0 & 1 & 0 \\
        0 & 0 & 0 & 0 & 0 & 0 & 1
       \end{smallmatrix} \right)
       d^{-1}
 =  \left( \begin{smallmatrix}
        1 & u^{-1}2r & 0 & 0 & 0 & 0 & 0 \\
        0 & 1 & 0 & 0 & 0 & 0 & 0 \\
        0 & 0 & 1 & 0 & 0 & 0 & u^{-1}r \\
        0 & 0 & 0 & 1 & 0 & -u^{-1}r & 0 \\
        -u^{-1}r & -u^{-2}r^2 & 0 & 0 & 1 & 0 & 0 \\
        0 & 0 & 0 & 0 & 0 & 1 & 0 \\
        0 & 0 & 0 & 0 & 0 & 0 & 1
       \end{smallmatrix} \right).
\]
Hence, $\phi(G(\tG,\eta,R)) = \phi(\mfX_\gamma(R) \rtimes \mcH_{\alpha}(R))$. 

The group $\mfX_\gamma(R) \rtimes \mcH_{\alpha}(R)$ above is in turn isomorphic to the affine group $\Aff_-(R)$, 
which is commensurable with $\BzeroR$ by Lemma~\ref{newlemma}. Thus, 
\[
 \phi(\mfX_\eta(R) \rtimes \mcH_\alpha(R)) = \phi(\Addi(R) \rtimes \Mult(R)) = \phi(\BzeroR)
\]
by Lemmata~\doubleref{obviousboundsonphi}{obphi4} and~\ref{newlemma}. This finishes the proof of the proposition and thus of Theorem~\ref{apendice}.
\end{proof}

\subsection{A geometric version of Theorem~\ref{apendice}} \label{RporQR}

We remark that Theorem~\ref{apendice} can be slightly modified as to avoid a representation $\rho~:~G \onto \mf{X}(R) \rtimes \mcH(R)$ with sequence $\ker(\rho) \into G \onto \mf{X}(R) \rtimes \mcH(R)$ split. We shall relax the hypothesis on $\rho$ at the cost of an assumption on the base ring $R$.

Recall that a metric space $Y$ is a \emph{quasi-retract} of a metric space $X$ if there exists a pair $r : X \to Y$ and $\iota : Y \to X$ of $(C,D)$-Lipschitz functions such that $d_Y(r \circ \iota (y), y) \leq D$ for all $y \in Y$; see~\cite{AlonsoFPnQI}. The point now is that quasi-retracts also inherit homotopical finiteness properties. In particular, if $G$ and $H$ are \emph{finitely generated} groups such that $r : G \to H$ is a quasi-retract, J.~M.~Alonso proved that $\phi(G) \leq \phi(H)$; see~\cite[Theorem~8]{AlonsoFPnQI}. Of course, group retracts are particular examples of quasi-retracts.

The advantage here is that a quasi-retract $r : G \to H$ needs not be a group homomorphism. In a recent remarkable paper, R.~Skipper, S.~Witzel and M.~Zaremsky~\cite{RachelStefanMatt} used the finiteness length and quasi-retracts to construct infinitely many quasi-isometry classes of finitely presented simple groups.

Using the geometric language above, we have the following.

\begin{thm}
  Let $G$ be a group and let $R$ be a commutative ring with unity for which $\Aff(R)$ is finitely generated. If there exists a quasi-retract $\rho : G \to \mfX(R) \rtimes \mcH(R)$, where $\mfX$ and $\mcH$ denote, respectively, a unipotent root subgroup and a maximal torus of any of the affine groups $\Aff$, $\Aff_-$ or of a classical group, then $\phi(G)\leq\phi(\BzeroR)$.
\end{thm}

\begin{proof}
 The assumption on $R$ implies that $\mf{X}(R) \rtimes \mcH(R)$ is finitely generated. Now, if $G$ is \emph{not} finitely generated, then $\phi(G) \leq \phi(\BzeroR)$ holds trivially. Otherwise one has $\phi(G) \leq \phi(\mfX(R) \rtimes \mcH(R))$ by Alonso's theorem~\cite[Theorem~8]{AlonsoFPnQI}. The fact that $\phi(\mfX(R)~\rtimes~\mcH(R)) \leq \phi(\BzeroR)$ was proved in Section~\ref{capum}, whence the theorem.
\end{proof}

\section{The Finiteness Lengths of Abels' Groups} \label{capdois}

In what follows we give a full proof of Theorem~\ref{Strebel's} and discuss open problems regarding finiteness properties of Abels' groups. 

We first observe that $\mbA_n(R)$ decomposes as a semi-direct product $\mbA_n(R) = \mbU_n(R) \rtimes \mbT_n(R)$, where
\[
 \mbU_n(R) =
 \left( \begin{smallmatrix}
           1 & * & \cdots & \cdots & * \\
            0 & 1 & \ddots & & \vdots \\
            \vdots & \ddots & \ddots & \ddots & \vdots \\
            0 & \cdots & 0 & 1 & * \\
            0 & \cdots & \cdots & 0 & 1
        \end{smallmatrix} \right),
 \mbT_n(R) =
  \left( \begin{smallmatrix}
           1 & 0 & \cdots & \cdots & 0 \\
            0 & * & \ddots & & \vdots \\
            \vdots & \ddots & \ddots & \ddots & \vdots \\
            0 & \cdots & 0 & * & 0 \\
            0 & \cdots & \cdots & 0 & 1
        \end{smallmatrix} \right) = \mbA_n(R) \cap \mbD_n(R).
\]
Just as with Abels' original group $\mbA_4(\Z[1/p])$ from~\cite{Abels0}, we see using relations~\eqref{Ancommutators} and~\eqref{SteinbergRelGLn} that the center of $\mbA_n(R)$ is the additive group $Z(\mbA_n(R)) = \mb{E}_{1n}(R) \cong \Addi(R)$ generated by all elementary matrices in the upper right corner. 

The first claims of Theorem~\ref{Strebel's} are well-known and follow from standard methods. For the sake of completeness, we shall also prove them in detail below in Section~\ref{provaAbels}. The tricky part of Theorem~\ref{Strebel's} is the claim~\eqref{nao-facil}. The outline of its proof is as follows. The first inequality follows from Theorem~\ref{apendice}, and it is also not hard to see that $\mbA_n(R)$ is finitely generated whenever $\BzeroR$ is so. Now, for a pair $(n,R)$ with $n \in \Z_{\geq 4}$, we construct a finite-dimensional connected simplicial complex $CC(\scrH(n,R))$ on which $\mbA_n(R)$ acts cocompactly by cell-permuting homeomorphisms. Generalizing a result due to S.~Holz, we show that the space $CC(\scrH(n,R))$ is simply-connected regardless of $R$. Using $\Sigma$-theory for metabelian groups~\cite{BieriStrebel}, we prove that all cell stabilizers of the given action $\mbA_n(R) \curvearrowright CC(\scrH(n,R))$ are finitely presented whenever $\BzeroR$ is so. We finish off the proof by invoking the following well-known criterion whose final form below is due to K.~S.~Brown. 

\begin{thm}[\cite{Brown0}] \label{Brownzinho}
 Let $G$ be a group acting by cell-permuting homeomorphisms on a connected, simply-connected CW-complex $X$ such that \emph{(a)} all vertex-stabilizers are finitely presented, \emph{(b)} all edge-stabilizers are finitely generated, and \emph{(c)} the $G$-action on the $2$-skeleton $X^{(2)}$ is cocompact. Then $G$ is finitely presented, that is, $\phi(G) \geq 2$.
\end{thm}

\subsection{A space for \texorpdfstring{$\mbA_n(R)$}{An(R)}}

Recall that a covering of a set $X$ is a collection of subsets $\set{X_\lambda}_{\lambda \in \Lambda}$ of $X$ whose union is the whole of $X$, i.e. $X = \cup_{\lambda \in \Lambda} X_\lambda$. The \emph{nerve of the covering} $\set{X_\lambda}_{\lambda \in \Lambda}$ is the simplicial complex ${N(\set{X_\lambda}_{\lambda \in \Lambda})}$ defined as follows. Its vertices are the sets $X_\lambda$ for $\lambda \in \Lambda$, and $k+1$ vertices $X_{\lambda_0}, X_{\lambda_1}, \ldots, X_{\lambda_k}$ span a $k$-simplex whenever the intersection of all such $X_{\lambda_i}$ is non-empty, i.e. $\cap_{i=0}^k X_{\lambda_i} \neq \leer$.

In~\cite{Holz,AbelsHolz}, Holz and Abels investigate nerve complexes attached to groups as follows. Fixing a family of subgroups, they take the nerve of the covering of the group by all cosets of subgroups of the given family. (Such spaces are also called \emph{coset posets} or \emph{coset complexes} in the literature.) More precisely, given a group $G$ and a family $\scrH = \set{H_\lambda}_{\lambda \in \Lambda}$ of subgroups of $G$, let $\mf{H}$ denote the covering $\mf{H} = \set{gH \mid g \in G, \, H \in \scrH}$ of $G$ by all left cosets of all members of $\scrH$. The coset complex ${CC(\scrH)}$ is defined as the nerve of the covering $N(\mf{H})$. In particular, if the family $\scrH$ is finite, then $CC(\scrH)$ is $(\vert\scrH\vert-1)$-dimensional.

Inspiration for the above came primarily from the theory of buildings: if $G$ is a group with a BN-pair $(G,B,N,S)$, then the coset poset $CC(\scrH)$ associated to the family $\scrH$ of all maximal standard parabolic subgroups of $G$ is by definition the building $\Delta(G,B)$ associated to the system $(G,B,N,S)$; see~\cite[Section~6.2]{AbraBrown} and \cite{BenjaminCM}. As it turns out, coset complexes show up in many other contexts, such as Deligne complexes~\cite{CharneyDavisKpi1Hyperplanes}, Bass--Serre theory~\cite{SerreTrees}, $\Sigma$-invariants of right-angled Artin groups~\cite{MMvW}, and higher generating families of braid groups~\cite{BFMWZ, BuxBraided} and automorphism groups of free groups~\cite{BenjaminBetween,BenjaminCM}.

Since vertices of $CC(\scrH)$ are cosets in the group $G$, it follows that $G$ acts naturally on $CC(\scrH)$ by cell-permuting homeomorphisms, namely the action induced by left multiplication on the cosets $gH$ for $g \in G$ and $H \in \scrH$.

Going back to Abels' groups, consider the following $\Z$-subschemes of $\mbA_n$.
\[
 H_1 = 
 \left( \begin{smallmatrix}
        1 & * & \cdots & *  & 0 \\
        0 & * & \ddots & \vdots  & \vdots \\
        \vdots & \ddots & \ddots  & * & \vdots \\
        0 &  & 0 & * & 0 \\
        0 & 0 & \cdots & 0  & 1
        \end{smallmatrix} \right),
        \, 
 H_2 =
  \left( \begin{smallmatrix}
        1 & 0 & \cdots & \cdots & 0 \\
        0 & * & * & \cdots  & * \\
        \vdots & \ddots & \ddots  & * & \vdots \\
        0 &  & 0 & * & * \\
        0 & 0 & \cdots & 0  & 1
        \end{smallmatrix} \right),
        \, 
 H_3 = 
  \left( \begin{smallmatrix}
        1 & * & 0 & \cdots & 0 \\
        0 & * & 0 & \ddots  & \vdots \\
        \vdots & \ddots & \ddots  & 0 & 0 \\
        0 &  & 0 & * & * \\
        0 & 0 & \cdots & 0  & 1
        \end{smallmatrix} \right).
\]
For $n = 4$ we consider, in addition, the following subscheme.
\[
 H_4 = 
 \begin{pmatrix}
         1 & 0 & * & 0 \\
         0 & * & * & * \\
         0 & 0 & * & 0 \\
         0 & 0 & 0 & 1
        \end{pmatrix} \leq \mbA_4.
\]
The unipotent radicals of the matrix groups above---i.e. the intersections of each $H_i$ with the group $\mbU_n$ of upper unitriangular matrices---are examples of group schemes arising from (maximal) contracting subgroups. 
Indeed, consider the locally compact group $\mbA_n(\K)$ for $\K$ a non-archimedean local field. In this case, each unipotent radical $\mcU_i(\K) = H_i(\K) \cap \mbU_n(\K)$ is the contracting subgroup associated to the automorphism given by conjugation by some element $t$ contained in the torus $\mbT_n(\K)$; see~\cite{Holz,Abels,UdoWillis0}. 
Holz shows~\cite[Sections~2.7.3 and~2.7.4]{Holz} that this defines a unipotent group scheme over $\Z$ depending on $t \in \mbT_n(\K)$. Following Abels we call the schemes $H_i$ above \emph{horospherical} and their unipotent radicals \, $\mcU_i = H_i \cap \mbU_n$ \emph{contracting} subgroups of $\mbA_n$.

For $R$ a commutative ring with unity and $n \geq 4$, let $\scrH(n, R)$ denote the family of groups of $R$-points of horospherical subgroups of $\mbA_n(R)$, i.e. 
\[
 {\scrH(n,R)} =
 \begin{cases}
  \set{H_1(R), \, H_2(R), \, H_3(R), \, H_4(R)}, & \mbox{if } n = 4;\\
  \set{H_1(R), \, H_2(R), \, H_3(R)} & \mbox{otherwise}.
 \end{cases}
\]
We also let
\[
 \mf{H}(n,R) = \set{gH \mid g \in \mbA_n(R), H \in \scrH(n,R)}.
\]
In the notation above, the space we shall consider is the nerve complex
\[
 CC(\scrH(n,R)) = N(\mf{H}(n,R))
\]
associated to the covering of $\mbA_n(R)$ by the left cosets $\mf{H}(n,R)$ of the horospherical subgroups listed above. As mentioned previously, the group $\mbA_n(R)$ acts on the simplicial complex $CC(\scrH(n,R))$ by cell-permuting homeomorphisms via left multiplication.

The space $CC(\scrH(n,R))$ has many useful features. Some of the facts we are about to list here actually hold more generally for arbitrary coset complexes. In the case of the following lemma, the properties that are particular to our groups come from the facts that the chosen family $\scrH(n,R)$ is finite and that the group $\mbA_n(R)$ is a split extension $\mbA_n(R) = \mbU_n(R) \rtimes \mbT_n(R)$ such that the contracting subgroups \, $\mcU_i(R) = H_i(R) \cap \mbU_n(R)$ (as well as intersections of contracting subgroups) are all $\mbT_n(R)$-invariant.

\begin{lem} \label{acaonub}
 The complex $CC(\scrH(n,R))$ is colorable and homogeneous, and the given $\mbA_n(R)$-action is type-preserving and cocompact. Any cell-stabilizer is isomorphic to a finite intersection of members of $\scrH(n,R)$.
\end{lem}

\begin{proof}
Since the intersection of cosets in a group is a coset of the intersection of the underlying subgroups, it follows that $CC(\scrH(n,R))$ is \emph{homogeneous}. That is to say, every simplex is contained in a simplex of dimension $k = \vert \scrH(n,R) \vert-1$ and every maximal simplex has dimension exactly $k$. (Note that $CC(\scrH(n,R))$ is even a chamber complex if $n \geq 5$.) We observe that $CC(\scrH(n,R))$ is colored, with \emph{types} (or \emph{colors}) given precisely by the family of subgroups $\scrH(n,R)$. Also, the given action of $\mbA_n(R)$ on $CC(\scrH(n,R))$ is type-preserving and transitive on the set of maximal simplices. Thus, the maximal simplex given by the intersection 
\[\bigcap\limits_{H \in \scrH(n,r)} H \]
is a fundamental domain for the $\mbA_n(R)$-action. In particular, since $\vert\scrH(n,R)\vert$ is finite, it follows that the action of $\mbA_n(R)$ is cocompact.

The stabilizers of the $\mbA_n(R)$-action are also not difficult to determine. For instance, given a maximal simplex $\sigma = \set{g_1 H_1(R), \, g_2 H_2(R), \, g_3 H_3(R)}$ in $CC(\scrH(n,R))$ with $n \geq 5$, there exists $g \in \mbA_n(R)$ such that $\sigma = g \cdot \set{H_1(R), \, H_2(R), \, H_3(R)}$. A group element $h \in \mbA_n(R)$ fixes $\sigma$ if and only if $h \in g \cdot (H_1(R) \cap H_2(R) \cap H_3(R)) \cdot g^{-1}$. A similar argument shows that a cell-stabilizer of $CC(\scrH(n,R))$ for any $n \geq 4$ is a conjugate of some intersection of subgroups that belong to the family $\scrH(n,R)$.
\end{proof}

The existence of a single simplex as fundamental domain, and the fact that cell-stabilizers are conjugates of finite intersections of members of a fixed family of subgroups, are properties that characterize coset complexes; see e.g.~\cite[Observation~A.4 and Proposition~A.5]{BFMWZ}.

Having determined the cell-stabilizers, we now prove that they are finitely presented whenever we need them to be.

\begin{pps} \label{stabilizersareF2}
 Suppose $\BzeroR$ is finitely presented and $n \geq 4$. Then any finite intersection of members of $\scrH(n,R)$ is finitely presented.
\end{pps}

\begin{proof}
We shall prove that, under the given assumption, the vertex-stabilizers are finitely presented. It will be clear from the arguments below that the same holds for stabilizers of higher dimensional cells. We remind the reader that the multiplicative group $\Mult(R)$, which is a retract of $\BzeroR$, is finitely presented because $\BzeroR$ itself is so; cf. Section~\ref{prelim}. In particular, $\mbT_n(R)$ is also finitely presented.

By Lemma~\ref{acaonub} we need only show that the members of $\scrH(n,R)$ are finitely presented. By the commutator~\eqref{Ancommutators} and diagonal relations~\eqref{SteinbergRelGLn}, we see that the `last-column subgroup' $\mc{C}_{n-1}$ of $H_1$, given by
\[
 \mc{C}_{n-1} =  
\left( \begin{smallmatrix}
        1 & 0 & \cdots & 0 & * & 0 \\
				0 & 1 & \ddots & \vdots & \vdots & \vdots \\
				\vdots & \ddots & \ddots & 0 & \vdots & \vdots \\
				\vdots &  & \ddots & 1 & * & \vdots \\
				 &  & & 0 & * & 0 \\
				0 & \cdots & & \cdots & 0 & 1
\end{smallmatrix} \right),
\]
is normal in $H_1$. The quotient $H_1 / \mc{C}_{n-1}$ is isomorphic to the subgroup
\[
 Q_{n-1} =  
\left( \begin{smallmatrix}
        1 & * & \cdots & * & 0 & 0 \\
				0 & * & \ddots & \vdots & \vdots & \vdots \\
				\vdots & \ddots & \ddots & * & \vdots & \vdots \\
				\vdots &  & \ddots & * & 0 & \vdots \\
				 &  & & 0 & 1 & 0 \\
				0 & \cdots & & \cdots & 0 & 1
\end{smallmatrix} \right) \leq H_1.
\]

We claim that the column subgroup $\mc{C}_{n-1}$ is itself finitely presented. To check this, let us briefly recall some concepts from the $\Sigma$-theory of Bieri--Strebel~\cite{BieriStrebel}. Given an abelian group $Q$ and a homomorphism $v : Q \to (\R,+)$, denote by $Q_v$ the monoid $Q_v = \set{q \in Q \mid v(q) \geq 0}$. We say that a $\Z[Q]$-module $A$ is \emph{tame} when $A$ is finitely generated over $\Z[Q_v]$ or over $\Z[Q_{-v}]$ (possibly both) for every homomorphism $v : Q \to (\R,+)$. One of the main results of~\cite{BieriStrebel} states that a finitely generated group $A \rtimes Q$ (with $A$ and $Q$ abelian) is finitely presented if and only if $A$ is a tame $\Z[Q]$-module. 

Going back to our proof, since $\BzeroR$ is assumed to be finitely presented, then so is the group 
\[\Aff_-(R) \cong \Addi(R) \rtimes_{-1} \Mult(R) \cong \left( \begin{smallmatrix} 1 & * \\ 0 & * \end{smallmatrix} \right) \leq \GL_2(R)\] 
by Lemma~\ref{newlemma}. In particular, the $\Z[\Mult(R)]$-module $\Addi(R)$ with the given action
\[
\Mult(R) \times \Addi(R) \ni (u,r) \mapsto u^{-1} r
\]
is tame by~\cite[Theorem~5.1]{BieriStrebel}. But $\mc{C}_{n-1}$ is seen to be isomorphic to $(\Addi(R))^{n-2} \rtimes \Mult(R)$, where the action of $\Mult(R)$ on each copy of $\Addi(R)$ is the reverse multiplication as above, and the action $\Mult(R) \curvearrowright \Addi(R)^{n-2}$ is just the diagonal action. Thus by~\cite[Proposition~2.5(i)]{BieriStrebel} it follows that $\Addi(R)^{n-2}$ is also a tame $\Z[\Mult(R)]$-module, whence one has---again by~\cite[Theorem~5.1]{BieriStrebel}---that $\mc{C}_{n-1} \cong \Addi(R)^{n-2} \rtimes \Mult(R)$ is finitely presented. 

We have shown that $H_1$ fits into a (split) short exact sequence
\[
 \mc{C}_{n-1} \into H_1 \onto Q_{n-1}
\]
where $\mc{C}_{n-1}$ is finitely presented. Decomposing $Q_{n-1}$ similarly via the last column, as we did with $H_1$, a simple induction on $n$ shows that $Q_{n-1}$ is also finitely presented. Lemma~\doubleref{obviousboundsonphi}{obphi0} thus shows that $H_1$ is finitely presented.

By considering the `first-row subgroup'
\[
 \mc{R}_{n-1} =  
\left( \begin{smallmatrix}
        1 & 0 &  & \cdots &  & 0 \\
				0 & * & * & \cdots & * & * \\
				\vdots & 0 & 1 & 0 & 0 & 0 \\
				\vdots &  & \ddots & \ddots & \ddots & \vdots \\
				 &  & & 0 & 1 & 0 \\
				0 & \cdots & & \cdots & 0 & 1
\end{smallmatrix} \right)
\]
of $H_2$, which is also normal by~\eqref{Ancommutators} and~\eqref{SteinbergRelGLn}, an argument analogous to the previous one, now using $\Aff(R) \cong \Addi(R) \rtimes \Mult(R) \cong \left( \begin{smallmatrix} * & * \\ 0 & 1 \end{smallmatrix} \right) \leq \GL_2(R)$, shows that $H_2$ is also finitely presented.

The case of $H_3$ is more straightforward since it equals the direct product
\[
\left( \begin{smallmatrix}
        1 & * & 0 & \cdots & 0 & 0 \\
				0 & * & 0 & \ddots & \vdots & \vdots \\
				\vdots & \ddots & \ddots & 0 & \vdots & \vdots \\
				\vdots &  & \ddots & * & 0 & \vdots \\
				 &  & & 0 & 1 & 0 \\
				0 & \cdots & & \cdots & 0 & 1
\end{smallmatrix} \right)
\times
\left( \begin{smallmatrix}
        1 & 0 & \cdots & 0 & 0 & 0 \\
				0 & 1 & \ddots & \vdots & \vdots & \vdots \\
				\vdots & \ddots & \ddots & \ddots & \vdots & \vdots \\
				\vdots &  & \ddots & 1 & 0 & \vdots \\
				 &  & & 0 & * & * \\
				0 & \cdots & & \cdots & 0 & 1
\end{smallmatrix} \right) \cong \Aff_-(R) \times \Mult(R)^{n-4} \times \Aff(R)
\]
and all factors on the right-hand side are finitely presented; see Lemma~\ref{newlemma} and Section~\ref{prelim}.

Establishing finite presentability of $H_4 \leq \mbA_4(R)$ is slightly different. Consider the subgroups
\[
\Gamma_1 =
 \left( \begin{smallmatrix}
        1 & 0 & * & 0 \\
				0 & * & * & 0 \\
				0 & 0 & * & 0 \\
				0 & 0 & 0 & 1
				\end{smallmatrix} \right)
\mbox{ and }
\Gamma_2 =
 \left( \begin{smallmatrix}
        1 & 0 & 0 & 0 \\
				0 & * & 0 & * \\
				0 & 0 & 1 & 0 \\
				0 & 0 & 0 & 1
				\end{smallmatrix} \right)
\]
and let $p_1 : \Gamma_1 \onto Q$ and $p_2 : \Gamma_2 \onto Q$ denote the natural projections onto the diagonal subgroup
\[
Q = 
\left( \begin{smallmatrix}
       1 & 0 & 0 & 0 \\
			 0 & * & 0 & 0 \\
			 0 & 0 & 1 & 0 \\
			 0 & 0 & 0 & 1
			\end{smallmatrix} \right).
\]
With this notation, we have that $H_4$ is isomorphic to the fiber product
\[
P = \set{(g,h) \in \Gamma_1 \times \Gamma_2 \mid p_1(g) = p_2(h)}. \index{Fiber product}
\]
We observe now that $\Gamma_1$ and $\Gamma_2$ are finitely presented---i.e. of homotopical type $\Fn{2}$---since
\[
 \ker(p_1) =
 \left( \begin{smallmatrix}
        1 & 0 & * & 0 \\
				0 & 1 & * & 0 \\
				0 & 0 & * & 0 \\
				0 & 0 & 0 & 1
				\end{smallmatrix} \right) \cong \mc{C}_{3} \leq H_1
\, \, \mbox{ and } \, \, 
 \left( \begin{smallmatrix}
        1 & 0 & 0 & 0 \\
				0 & * & 0 & * \\
				0 & 0 & 1 & 0 \\
				0 & 0 & 0 & 1
				\end{smallmatrix} \right) \cong \Aff(R) \leq \GL_2(R)
\]
are so. (In particular, $\ker(p_1)$ is of type $\Fn{1}$.) Again, the finite presentability of $\BzeroR$ implies that the abelian group $Q \cong \Mult(R)$ is finitely generated, whence it is of type $\Fn{3}$; cf. Section~\ref{prelim}. Therefore the fiber product $P \cong H_4$ is finitely presented by the (asymmetric) 1-2-3-Theorem~\cite[Theorem~B]{BHMS3}.

Entirely analogous arguments for the non-trivial intersections of members of $\scrH(n,R)$ show that all such groups are also finitely presented, which yields the proposition.
\end{proof}

We now investigate connectivity properties of the complex $CC(\scrH(n,R))$. The following observation, whose proof we omit, is originally due to Holz~\cite{Holz}. To verify it directly, consider the homotopy equivalences given in~\cite[Theorem~1.4]{AbelsHolz}. Alternatively, a generalization has been recently obtained by B.~Br\"uck; see \cite[Theorem~3.17 and Corollary~3.18]{BenjaminBetween} for a proof.

\begin{lem}[{\cite[Korollar~5.18]{Holz}}] \label{semidirectHolz}
Let $G = N \rtimes Q$ and suppose $\scrH$ is a family of $Q$-invariant subgroups of $N$. Then there exists a homotopy equivalence between the coset complex $CC(\scrH)$ of $\scrH$ \emph{with respect to} $N$ and the coset complex $CC(\set{H \rtimes Q \mid H \in \scrH})$ \emph{with respect to overgroup} $G$.
\end{lem}

For our groups, Lemma~\ref{semidirectHolz} yields the following.

\begin{cor} \label{corollaryHolz}
  Let $\scrH_u(n,R)$ denote the family of unipotent radicals $\mcU_i(R) = H_i(R) \cap \mbU_n(R)$ of members of $\scrH(n,R)$ and consider the following covering of $\mbU_n(R)$.
  \[
   \mf{H}_u(n,R) = \set{v \cdot \mcU(R) \mid v \in \mbU_n(R) \mbox{ and } \, \mcU(R) \in \scrH_u(n,R)}.
  \]
Then the spaces $CC(\scrH_u(n,R)) = N(\mf{H}_u(n,R))$ (with respect to $\mbU_n(R)$) and $CC(\scrH(n,R))$ (with respect to $\mbA_n(R)$) are homotopy equivalent.
\end{cor}

\begin{proof}
 This follows at once from Lemma~\ref{semidirectHolz} since the $\mbT_n(R)$-action by conjugation preserves each $\mcU_i(R)$ by the diagonal relations~\eqref{SteinbergRelGLn}.
\end{proof}

Thus, to show that $CC(\scrH(n,R))$ is connected and simply-connected, it suffices to prove that the coset complex $CC(\scrH_u(n,R))$ of contracting subgroups, with cosets taken \emph{in the unipotent radical} $\mbU_n(R)$, is connected and simply-connected. To do so we take advantage of the algebraic meaning of connectivity properties of coset complexes discovered by Tits~\cite{TitsAmalgams}.

Recall that the colimit $\mathrm{colim}\, F$ of a diagram $F : I \to \mathrm{Grp}$ from a small category $I$ to the category of groups is a group $K$ together with a family of maps $\Psi = \set{\psi_O : F(O) \to K}_{O \in \mathrm{Obj}(I)}$ satisfying the following properties.

\begin{itemize}
 \item $\psi_P \circ F(f) = \psi_O$ for all $f \in \Hom(O, P)$;
 \item If $(K', \Psi')$ is another pair also satisfying the conditions above, then there exists a unique group homomorphism $\phee : K \to K'$ such that $\phee \circ \psi_O = \psi_O'$ for all $O \in \mathrm{Obj}(I)$.
\end{itemize}

In this case we write $K = \mathrm{colim}\, F$, omitting the maps $\Psi$. 
Now suppose $\scrH$ is a family of subgroups of a given group. This induces a diagram $F_\scrH : I_\scrH \to \mathrm{Grp}$ by defining the category $I_\scrH$ to be the poset given by members of $\scrH$ and their pairwise intersections, ordered by inclusion. For example, if $\scrH = \set{A,B}$ with $A, B \leq G$ and $C = A \cap B$, then $F_\scrH$ is just the usual diagram $A \hookleftarrow C \hookrightarrow B$ 
and thus the colimit $\mathrm{colim}\, F_\scrH$ is simply the push-out (or amalgamated product) $\mathrm{colim}\, F_\scrH = A \ast_C B$.

\begin{thm}[{\cite{TitsAmalgams}; see also \cite[Theorem~2.4]{AbelsHolz}}] \label{AbelsHolzThmCC} 
 Let $\scrH$ be a family of subgroups of a group $G$ and let $\pi : \mathrm{colim}\, F_\scrH \to G$ denote the natural map from the colimit of $F_\scrH$ to $G$. Then the coset complex $CC(\scrH)$ is connected if and only if $\pi$ is surjective, and $CC(\scrH)$ is additionally simply-connected if and only if $\pi$ is an isomorphism.
\end{thm}

To apply Theorem~\ref{AbelsHolzThmCC} in our context we will need a bit of commutator calculus. The following identities are well-known; see e.g.~\cite[Section~2.2]{Abels}. 

\begin{lem} \label{commids} 
Let $G$ be a group and let $a,b,c \in G$. Then
\begin{equation} \label{comutinho}
[ab,c] = a[b,c]a^{-1}[a,c]
\end{equation}
and
\begin{equation} \label{Hall}
 [cac^{-1},[b,c]] \cdot [bcb^{-1},[a,b]] \cdot [aba^{-1},[c,a]] = 1. \tag{Hall's identity}
\end{equation}
\end{lem}

We also need a convenient, well-known presentation for $\mbU_n(R)$. To describe this standard presentation we need some notation. Fix $T \subseteq R$ a generating set, containing the unit $1$, for the underlying additive group $\Addi(R)$ of the base ring $R$. That is, we view $R$ as a quotient of the free abelian group $\bigoplus\limits_{t \in T} \Z t$. 

We fix furthermore $\mc{R} \subseteq \bigoplus_{t \in T} \Z t$ a set of \emph{additive} defining relations of $R$. In other words, $\mc{R}$ is a set of expressions $\set{\sum\limits_{\ell} a_\ell t_\ell \mid a_\ell \in \Z, t_\ell \in T} \subseteq \bigoplus\limits_{t \in T} \Z t$ (where all but finitely many coefficients $a_\ell$ are zero) such that $\Addi(R) \cong \left(\bigoplus\limits_{t \in T} \Z t\right) / \spans{\mc{R}}_\Z$.

For every pair $t,s\in T$ of additive generators, we choose an expression $m(t,s) = m(s,t) \in \bigoplus_{t \in T} \Z t$ such that the image of $m(s,t)$ in $R$ under the given projection $\bigoplus_{t \in T} \Z t \onto R$ equals the products $ts$ and $st$. In case $t = 1$, we take $m(1,s)$ to be $s$ itself, i.e. $m(1,s) = s = m(s,1)$.

\begin{lem} \label{standardpresUn} 
 With the notation above, the group $\mbU_n(R) \leq \mbA_n(R)$ admits a presentation $\mbU_n(R) \cong \spans{\mc{Y} \mid \mc{S}}$ with generating set
\[
\mc{Y} = \set{\eij(t) \mid t\in T, 1 \leq i < j \leq n},
\]
and a set of defining relations $\mc{S}$ given as follows. For all $(i,j)$ with $1 \leq i < j \leq n$ and all pairs $t, s \in T$,
\begin{equation} \label{UnRel.1}
[\ekl{ij}(t),\ekl{kl}(s)]  = \begin{cases} \prod\limits_u \ekl{il}(u)^{a_u}, & \mbox{ if } j=k;\\ 
1, & \mbox{ if } i \neq l, k \neq j, \end{cases}
\end{equation}
where $m(t,s) = \sum\limits_{u} a_u u \in \bigoplus\limits_{t \in T} \Z t$ is the fixed expression $m(t,s)$ as above for the product $ts = st \in R$. \\
For all $(i,j)$ with $1 \leq i < j \leq n$,
\begin{equation} \label{UnRel.2}
\prod\limits_{\ell} \eij(t_{\ell})^{a_\ell} = 1 \mbox{ for each } \sum\limits_{\ell} a_\ell t_{\ell} \in \mc{R}.
\end{equation}
The set $\mc{S}$ is defined as the set of all relations~\eqref{UnRel.1} and~\eqref{UnRel.2} given above.
\end{lem}

Lemma~\ref{standardpresUn} is far from new, so we omit its proof. The presentation above has been considered many times in the literature, most notably in the case where $R$ is a field and in connection to buildings and amalgams; see e.g.~\cite[Chapter~3]{Steinberg}, \cite{TitsAmalgams}, \cite[Appendix~2]{TitsBN}, \cite{DevillersMuehlherr}, and~\cite[Chapters~7 and~8]{AbraBrown}. In general, such presentation is extracted from the commutator relations~\eqref{Ancommutators} between elementary matrices---recall that, in the Chevalley--Demazure set-up, $\mbU_n$ is a maximal unipotent subscheme (over $\Z$) in type $\tA_{n-1}$. The only difference between the presentation we spelled out and other versions typically occurring elsewhere is that we made the ring structure of $R$ more explicit in the relations occurring in $\mbU_n(R)$. The interested reader is referred e.g. to~\cite[Section~1.1.2]{YuriThesis} for a detailed proof of Lemma~\ref{standardpresUn}.

Using the above results, we shall have the last ingredient for the proof of Theorem~\doubleref{Strebel's}{nao-facil} once we establish the following generalization of a result due to Holz~\cite[Proposition~A.3]{Holz}. 

\begin{pps} \label{Holzs}
For every $n \geq 4$ one has that $\mbU_n(R)\cong\mathrm{colim}~F_{\scrH_u(n,R)}$.
\end{pps}

\begin{proof}
 The idea is to write down a convenient presentation for $\mbU_n(R)$ which shows that it is the desired colimit. To do so, we first spell out canonical presentations for the members of $\scrH_u(n,R)$. For the course of this proof we fix (and follow strictly) the notation of Lemma~\ref{standardpresUn}. In particular, $T \subseteq R$ will denote an arbitrary, but fixed, additive generating set for $(R,+) = \Addi(R)$ containing $1$. As in Lemma~\ref{standardpresUn}, we fix $\mc{R}$ a set of additive defining relations of $\Addi(R)$.
 
 We observe that $\mcU_3(R)$ and $\mcU_4(R)$ are abelian, by the commutator relations~\eqref{Ancommutators}. It is also not hard to see that $\mcU_1(R) \cong \mbU_{n-1}(R) \cong \mcU_2(R)$ by translating the indices of elementary matrices accordingly. Thus, we have the following presentations.
 \begin{align*}
  \mcU_1(R) \cong & \langle \set{\eij(t) \mid t \in T, \, 1 \leq i < j \leq n-1} \mid \mbox{ Relations \eqref{UnRel.1} and \eqref{UnRel.2}} \\ 
  & \mbox{for all } i,j \mbox{ with } 1 \leq i < j \leq n-1, \mbox{ and all } t,s \in T \rangle.
 \end{align*}
  \begin{align*}
  \mcU_2(R) \cong & \langle \set{\eij(t) \mid t \in T, 2 \leq i < j \leq n} \mid \mbox{ Relations \eqref{UnRel.1} and \eqref{UnRel.2}} \\ 
  & \mbox{for all } i,j \mbox{ with } 2 \leq i < j \leq n, \mbox{ and all } t,s \in T \rangle.
 \end{align*}
  \begin{align*}
  \mcU_3(R) \cong & \langle \set{\ekl{12}(t), \ekl{n-1,n}(s) \mid t,s \in T} \mid [\ekl{12}(t), \ekl{n-1,n}(s)] = 1 \mbox{ for all } t,s \in T, \\
  & \mbox{and relations \eqref{UnRel.2} for all } t,s \in T \mbox{ and } (i,j) \in \set{(1,2), (n-1,n)} \rangle.
 \end{align*}
  \begin{align*}
  \mcU_4(R) \cong & \langle \set{\ekl{13}(p), \, \ekl{23}(t), \, \ekl{24}(s) \mid p,t,s \in T} \, \mid \,  [\ekl{ij}(t), \ekl{kl}(s)] = 1 \mbox{ for all } \\
  & t, s \in T \mbox { and } (i,j), (k,l) \in \set{(1,3),(2,3),(2,4)}, \mbox{ and} \\
  & \mbox{relations \eqref{UnRel.2} for all } t,s \in T \mbox{ and } (i,j) \in \set{(1,3),(2,3),(2,4)} \rangle.
 \end{align*}
  The pairwise intersections $\mcU_i(R) \cap \mcU_j(R)$ also admit similar presentations by restricting the generators (and corresponding relations) to the indices occurring in both $\mcU_i(R)$ and $\mcU_j(R)$. For instance, 
  \[
   \mcU_1(R) \cap \mcU_2(R) =
   \left( \begin{smallmatrix}
           1 & 0 & \cdots & \cdots & 0 & 0 \\
           0 & 1 & * & \cdots & * & 0 \\
            \vdots & \ddots & \ddots & \ddots & \vdots & \vdots \\
            &  & 0 & 1 & * & 0 \\
            &  &  & 0 & 1 & 0 \\
           0 & \cdots &  & \cdots & 0 &  1
          \end{smallmatrix} \right)
   \cong \mbU_{n-2}(R),
  \]
  with presentation
   \begin{align*}
  \mcU_1(R) \cap \mcU_2(R) \cong & \langle \set{\eij(t) \mid t \in T, \, 2 \leq i < j \leq n-1} \mid \mbox{ Relations \eqref{UnRel.1} and} \\ 
  & \mbox{\eqref{UnRel.2} for all } i,j \mbox{ with } 2 \leq i < j \leq n-1, \mbox{ and all } t,s \in T \rangle.
 \end{align*}
 
Now consider the group $U_n$ defined as follows. As generating set we take
 \[
  \mc{X}_n = \set{\eij(t), \, \ekl{kn}(s) \, \mid \, 1 \leq i < j \leq n-1, \, \, 2 \leq k \leq n, \mbox{ and } t,s \in T}.
 \]
 The set of defining relations $\mc{S}_n$ is formed as follows. For all $t,s \in T$ and indices $i,j,k,l$ which are \emph{either all} in $\set{1,\ldots,n-1}$ \emph{or all} in $\set{2,\ldots,n}$, consider the relations
 \begin{equation} \label{economicUnRel.1}
[\ekl{ij}(t),\ekl{kl}(s)]  = \begin{cases} \prod\limits_u \ekl{il}(u)^{a_u}, & \mbox{ if } j=k;\\ 
1, & \mbox{ if } i \neq l, k \neq j, \end{cases}
\end{equation}
and
 \begin{equation} \label{economicUnRel.2}
  [\ekl{12}(t),\ekl{n-1,n}(s)] = 1,
 \end{equation}
where $m(t,s) = \sum\limits_{u} a_u u \in \bigoplus\limits_{t \in T} \Z t$ is as in Lemma~\ref{standardpresUn}. 
For all $t,s \in T$ and pairs $i,j$ which are \emph{either all} in $\set{1,\ldots,n-1}$ \emph{or all} in $\set{2,\ldots,n}$, consider additionally the relations 
\begin{equation} \label{economicUnRel.3}
\prod_{\ell=1}^m \eij(t_{\ell})^{a_\ell} = 1 \mbox{ for each } \sum_{\ell=1}^m a_\ell t_{\ell} \in \mc{R},
\end{equation}
where $\mc{R}$ is the fixed set of additive defining relations of $\Addi(R)$ as in Lemma~\ref{standardpresUn}. If $n = 4$ we need also consider the relations
 \begin{equation} \label{economicUnRel.4}
  [\ekl{13}(t),\ekl{24}(s)] = 1
 \end{equation}
 for all pairs $t,s \in T$. We take $\mc{S}_n$ to be the set of all relations \eqref{economicUnRel.1}, \eqref{economicUnRel.2}, and \eqref{economicUnRel.3} (in case $n \geq 5$), and $\mc{S}_4$ is the set of all relations \eqref{economicUnRel.1} through \eqref{economicUnRel.4} above. We then define $U_n$ by means of the presentation
\[
 U_n = \spans{\mc{X}_n \mid \mc{S}_n}.
\]

Reading off the presentations for the $\mcU_i(R)$ and for their pairwise intersections, it follows from von Dyck's theorem that $\mathrm{colim}\, F_{\scrH_u(n,R)}$ is isomorphic to the group $U_n$ above.

Thus, to prove the proposition it suffices to show that $\mbU_n(R)$ is isomorphic to $U_n$. To avoid introducing even more symbols 
we proceed as follows. Recall that $\mbU_n(R)$ admits the presentation $\mbU_n(R) \cong \spans{\mc{Y} \mid \mc{S}}$ given in Lemma~\ref{standardpresUn}. Abusing notation and comparing the presentations $U_n = \spans{\mc{X}_n \mid \mc{S}_n}$ and $\mbU_n(R) \cong \spans{\mc{Y} \mid \mc{S}}$, it suffices to define in $U_n$ the missing generators $\ekl{1n}(t)$ (for $t \in T$) and also show that all the relations from $\mc{S}$ missing from $\mc{S}_n$ do hold in $U_n$. (Inspecting the indices, the missing relations are those involving the commutators $[\ekl{1j}(t),\ekl{kn}(s)]$ for $j = 2, \ldots, n$ and $k = 1, \ldots, n-1$, and $(j,k) \neq (2,n-1)$ and those involving only the new generators $\ekl{1n}(t)$.)

For every $t \in T$, define in $U_n$ the element $\ekl{1n}(t) = [\ekl{12}(t), \ekl{2n}(1)]$. With this new commutator at hand, the proof will be concluded once we show that the following equalities hold in $U_n$.\\
For all $s,t \in T$ and $j,k \in \set{2,\ldots,n-1}$ with $j \neq k$ and $(j,k) \neq (2,n-1)$,
\begin{equation} \label{missingrel1}
 [\ekl{1j}(t),\ekl{kn}(s)] = 1.
\end{equation} 
For all $t \in T$ and $j \in \set{2,\ldots,n-1}$,
\begin{equation} \label{missingrel2}
 [\ekl{1j}(t),\ekl{jn}(1)] = [\ekl{1j}(1),\ekl{jn}(t)] = \ekl{1n}(t).
\end{equation}
For all $t, s \in T$ and $i,j$ with $1 \leq i < j \leq n$,
\begin{equation} \label{missingrel3}
 [\ekl{ij}(t),\ekl{1n}(s)] = 1.
\end{equation}
For all $\sum_{\ell=1}^m a_\ell t_{\ell} \in \mc{R}$,
\begin{equation} \label{missingrel4}
 \prod_{\ell=1}^m \ekl{1n}(t_{\ell})^{a_\ell} = 1.
\end{equation}

{\bf Relation~(\ref{missingrel1}) holds:} If $n = 4$ there is nothing to show, since in this case the only equation to verify is $[\ekl{13}(t),\ekl{24}(s)] = 1$, which holds by~\eqref{economicUnRel.4}. Assume $n \geq 5$. We first observe that
\begin{equation} \label{contasUn1}
 [\ekl{1j}(t), \ekl{n-1,n}(s)] = 1
\end{equation}
for all $t,s \in T$ and $j \in \set{3,\ldots,n-2}$ since
\begin{align*}
 \ekl{1j}(t) \ekl{n-1,n}(s) & \overset{\mbox{\scriptsize\eqref{economicUnRel.1}}}{=} \ekl{12}(t) \ekl{2j}(1) \ekl{12}(t)^{-1} \ekl{2j}(1)^{-1} \ekl{n-1,n}(s) \\
 & \overset{\mbox{\scriptsize\eqref{economicUnRel.1}}}{=} \ekl{12}(t) \ekl{2j}(1) \ekl{12}(t)^{-1} \ekl{n-1,n}(s) \ekl{2j}(1)^{-1} \\
 & \overset{\mbox{\scriptsize\eqref{economicUnRel.2}}}{=} \ekl{12}(t) \ekl{2j}(1) \ekl{n-1,n}(s) \ekl{12}(t)^{-1} \ekl{2j}(1)^{-1} \\
 & \overset{\mbox{\scriptsize\eqref{economicUnRel.1}}}{=} \ekl{12}(t) \ekl{n-1,n}(s) \ekl{2j}(1) \ekl{12}(t)^{-1} \ekl{2j}(1)^{-1} \\
 & \overset{\mbox{\scriptsize\eqref{economicUnRel.2}}}{=} \ekl{n-1,n}(s) \ekl{12}(t) \ekl{2j}(1) \ekl{12}(t)^{-1} \ekl{2j}(1)^{-1} \\
 & \overset{\mbox{\scriptsize\eqref{economicUnRel.1}}}{=} \ekl{n-1,n}(s) \ekl{1j}(t).
\end{align*}
Proceeding similarly, we conclude that
\begin{equation} \label{contasUn2}
 [\ekl{12}(t), \ekl{kn}(s)] = 1
\end{equation}
for all $t, s \in T$ and $k \in \set{3,\ldots,n-2}$. Now suppose $j < k$. Then
\[
 [\ekl{1j}(t),\ekl{kn}(s)] \overset{\mbox{\scriptsize\eqref{economicUnRel.1}}}{=} [\ekl{1j}(t), [\ekl{k,n-1}(s),\ekl{n-1,n}(1)]] = 1
\]
because $\ekl{1j}(t)$ commutes with $\ekl{n-1,n}(1)$, by~\eqref{contasUn1}, and with $\ekl{k,n-1}(s)$, by~\eqref{economicUnRel.1}. Analogously, if $j > k$, then
\[
 [\ekl{1j}(t),\ekl{kn}(s)] \overset{\mbox{\scriptsize\eqref{economicUnRel.1}}}{=} [[\ekl{12}(t),\ekl{2j}(1)],\ekl{kn}(s)] = 1
\]
by~\eqref{economicUnRel.1} and~\eqref{contasUn2}. Thus, the relations~\eqref{missingrel1} hold in $U_n$.\\

{\bf Relation~(\ref{missingrel2}) holds:} To check~\eqref{missingrel2} we need Lemma~\ref{commids}. First,
\[
 [\ekl{12}(t),\ekl{2n}(1)] \overset{\mbox{\scriptsize\eqref{economicUnRel.1}}}{=} [\ekl{12}(t),[\ekl{23}(1),\ekl{3n}(1)]].
\]
Setting $a = \ekl{12}(t)$, $b = \ekl{23}(1)$, and $c = \ekl{3n}(1)$, \ref{Hall} yields
\begin{align*}
 1 & = [cac^{-1},[b,c]] \cdot [bcb^{-1},[a,b]] \cdot [aba^{-1},[c,a]] \\
  & \overset{\mbox{\scriptsize\eqref{economicUnRel.1}}}{=} [\ekl{12}(t),\ekl{2n}(1)] \cdot [\ekl{23}(1) \ekl{3n}(1) \ekl{23}(1)^{-1}, \ekl{13}(t)] \\
  & \overset{\mbox{\scriptsize\eqref{economicUnRel.1}}}{=} [\ekl{12}(t),\ekl{2n}(1)] \cdot [\ekl{2n}(1) \ekl{3n}(1), \ekl{13}(t)] \\
  & \overset{\mbox{\scriptsize\eqref{comutinho}}}{=} [\ekl{12}(t),\ekl{2n}(1)] \cdot \ekl{2n}(1) \cdot [\ekl{3n}(1),\ekl{13}(t)] \cdot \ekl{2n}(1)^{-1} \cdot [\ekl{2n}(1),\ekl{13}(t)] \\
  & \overset{\mbox{\scriptsize\eqref{economicUnRel.1}\&\eqref{missingrel1}}}{=} [\ekl{12}(t),\ekl{2n}(1)] \cdot [\ekl{3n}(1),\ekl{13}(t)],
\end{align*}
that is, $\ekl{1n}(t) = [\ekl{13}(t),\ekl{3n}(1)]$. On the other hand,
\[
 [\ekl{12}(1),\ekl{2n}(t)] \overset{\mbox{\scriptsize\eqref{economicUnRel.1}}}{=} [\ekl{12}(1),[\ekl{23}(t),\ekl{3n}(1)]].
\]
Setting $a = \ekl{12}(1)$, $b = \ekl{23}(t)$, and $c = \ekl{3n}(1)$, \ref{Hall} and~\eqref{economicUnRel.1} yield
\begin{align*}
 1 & {=} [\ekl{12}(1),\ekl{2n}(t)] \cdot [\ekl{23}(t) \ekl{3n}(1) \ekl{23}(1)^{-1}, \ekl{13}(t)] \\
  & \overset{\mbox{\scriptsize\eqref{economicUnRel.1}}}{=} [\ekl{12}(1),\ekl{2n}(t)] \cdot [\ekl{2n}(t) \ekl{3n}(1), \ekl{13}(t)] \\
  & \overset{\mbox{\scriptsize\eqref{comutinho}}}{=} [\ekl{12}(1),\ekl{2n}(t)] \cdot \ekl{2n}(t) \cdot [\ekl{3n}(1),\ekl{13}(t)] \cdot \ekl{2n}(t)^{-1} \cdot [\ekl{2n}(t),\ekl{13}(t)] \\
  & \overset{\mbox{\scriptsize\eqref{economicUnRel.1}\&\eqref{missingrel1}}}{=} [\ekl{12}(1),\ekl{2n}(t)] \cdot [\ekl{3n}(1),\ekl{13}(t)].
\end{align*}
The last product above equals $[\ekl{12}(1),\ekl{2n}(t)] \ekl{1n}(t)^{-1}$ by the previous computations. We have thus proved that
\[
\ekl{1n}(t) \overset{\mbox{\scriptsize {Def.}}}{=} [\ekl{12}(t),\ekl{2n}(1)] = [\ekl{12}(1),\ekl{2n}(t)] = [\ekl{13}(t),\ekl{3n}(1)].
\]
Since $[\ekl{12}(1),\ekl{2n}(t)]$ also equals  $[\ekl{12}(1),[\ekl{23}(1),\ekl{3n}(t)]]$, again by~\eqref{economicUnRel.1}, computations similar to the above also yield $[\ekl{13}(t),\ekl{3n}(1)] = [\ekl{13}(1),\ekl{3n}(t)]$. Entirely analogous arguments show that
\[
[\ekl{1j}(t),\ekl{jn}(1)] = [\ekl{1j}(1),\ekl{jn}(t)] = \ekl{1n}(t)
\]
for all $j \in \set{2, \ldots, n-1}$. \\

{\bf Relations~(\ref{missingrel3}) hold:} We now prove that the subgroup $Z :=~\spans{\,\set{\ekl{1n}(t) \mid t \in T}\,} \leq U_n$ is central. Let $t, s \in T$ and let $i,j$ be such that $1 \leq i < j \leq n$. We want to show that $\ekl{1n}(t)$ and $\eij(s)$ commute in $U_n$. To begin with,
\[
[\ekl{1n}(t),\ekl{1n}(s)] \overset{\mbox{\scriptsize\eqref{missingrel2}}}{=} [[\ekl{12}(t),\ekl{2n}(1)],[\ekl{13}(s),\ekl{3n}(1)]] \overset{\mbox{\scriptsize\eqref{economicUnRel.1}}}{=} 1,
\]
 i.e. $Z$ is abelian. If $i = 1$ and $j \neq n$, then $j \geq 2$ and we can pick $k \in~\set{2, \ldots, n-1}$ such that $k \neq j$ because $n \geq 4$, yielding
\begin{align*}
\ekl{1n}(t) \ekl{1j}(s) \overset{\mbox{\scriptsize\eqref{missingrel2}}}{=} [\ekl{1k}(t), \ekl{kn}(1) ] \ekl{1j}(s) & \overset{\mbox{\scriptsize\eqref{missingrel1}\&\eqref{economicUnRel.1}}}{=} \ekl{1j}(s) [\ekl{1k}(t), \ekl{kn}(1)] \\
 & \overset{\mbox{\scriptsize\phantom{111}\eqref{missingrel2}\phantom{x11}}}{=} \ekl{1j}(s) \ekl{1n}(t).
\end{align*}
Similarly, if $j = n$ and $i \neq 1$, choose $k \in \set{2, \ldots, n-1}$ such that $k \neq i$. We obtain
\[
[\ekl{1n}(t),\ekl{in}(s)] \overset{\mbox{\scriptsize\eqref{missingrel2}}}{=} [[\ekl{1k}(t),\ekl{kn}(1)],\ekl{in}(s)] \overset{\mbox{\scriptsize\eqref{economicUnRel.1}\&\eqref{missingrel1}}}{=} 1.
\]
It remains to prove $[\ekl{1n}(t), \eij(s)] = 1$ for $1 < i < j < n$. In this case, 
\begin{align*}
\ekl{1n}(t) \eij(s) & \overset{\mbox{\scriptsize\phantom{x11}\eqref{missingrel2}\phantom{x11}}}{=} \ekl{1i}(1) \ekl{in}(t) \ekl{1i}(1)^{-1} \ekl{in}(t)^{-1} \eij(s) \\
& \overset{\mbox{\scriptsize\phantom{x11}\eqref{economicUnRel.1}\phantom{x11}}}{=} \ekl{1i}(1) \ekl{in}(t) \ekl{1i}(1)^{-1}  \eij(s) \ekl{in}(t)^{-1} \\
& \overset{\mbox{\scriptsize\eqref{economicUnRel.3}\&\eqref{economicUnRel.1}}}{=} \ekl{1i}(1) \ekl{in}(t) \ekl{1j}(s)^{-1} \eij(s) \ekl{1i}(1)^{-1} \ekl{in}(t)^{-1} \\
& \overset{\mbox{\scriptsize\eqref{missingrel1}\&\eqref{economicUnRel.1}}}{=} \ekl{1j}(s)^{-1} \ekl{1i}(1) \ekl{in}(t) \eij(s) \ekl{1i}(1)^{-1} \ekl{in}(t)^{-1} \\
& \overset{\mbox{\scriptsize\phantom{x11}\eqref{economicUnRel.1}\phantom{x11}}}{=} \ekl{1j}(s)^{-1} \ekl{1i}(1) \eij(s) \ekl{in}(t) \ekl{1i}(1)^{-1} \ekl{in}(t)^{-1} \\
& \overset{\mbox{\scriptsize\eqref{economicUnRel.3}\&\eqref{economicUnRel.1}}}{=} \ekl{1j}(s)^{-1} \ekl{1j}(s) \eij(s) \ekl{1i}(1) \ekl{in}(t) \ekl{1i}(1)^{-1} \ekl{in}(t)^{-1} \\
& \overset{\mbox{\scriptsize\phantom{x11}\eqref{missingrel2}\phantom{x11}}}{=} \eij(s) \ekl{1n}(t).
\end{align*}
Thus, relations~\eqref{missingrel3} hold for all $t, s \in T$ and $i,j$ with $1 \leq i < j \leq n$.\\

{\bf Relations~(\ref{missingrel4}) hold:} Given any pair $t,s \in T$, 
\begin{align} 
\nonumber [\ekl{12}(t) \ekl{12}(s), \ekl{2n}(1)] & \overset{\mbox{\scriptsize\phantom{x11}\eqref{comutinho}\phantom{x11}}}{=} \ekl{12}(t)[\ekl{12}(s), \ekl{2n}(1)] \ekl{12}(t)^{-1} [\ekl{12}(t), \ekl{2n}(1)] \\
\nonumber & \overset{\mbox{\scriptsize\eqref{missingrel2}\&\eqref{missingrel3}}}{=} [\ekl{12}(s), \ekl{2n}(1)] [\ekl{12}(t), \ekl{2n}(1)] \\
 \label{ultimoproduto} & \overset{\mbox{\scriptsize\eqref{missingrel2}\&\eqref{missingrel3}}}{=} [\ekl{12}(t), \ekl{2n}(1)] [\ekl{12}(s), \ekl{2n}(1)].
\end{align}
Now let $\sum\limits_{\ell = 1}^m a_\ell t_\ell \in \mc{R}$ be an additive defining relation in $R$. (Recall that $t_\ell \in T$ and $a_\ell \in \Z$ as in Lemma~\ref{standardpresUn}.) Induction on $\sum\limits_{\ell = 1}^{m} |a_\ell|$ and~\eqref{ultimoproduto} yield 
\[
\prod\limits_{\ell = 1}^{m} \ekl{1n}(t_\ell)^{a_\ell} \overset{\mbox{\scriptsize\mbox{Def.}}}{=} \prod\limits_{\ell = 1}^{m} \left([\ekl{12}(t_\ell), \ekl{2n}(1)]\right)^{a_\ell} \overset{\mbox{\scriptsize\eqref{ultimoproduto}}}{=} \left[ \prod\limits_{\ell = 1}^{m} \ekl{12}(t_\ell)^{a_\ell}, \ekl{2n}(1) \right] \overset{\mbox{\scriptsize\eqref{economicUnRel.1}}}{=} 1. 
\]

Since the relations \eqref{missingrel1} -- \eqref{missingrel4} missing from the presentation for $\mbU_n(R)$ from Lemma~\ref{standardpresUn} also hold in $U_n$, it follows that $U_n  \cong \mbU_n(R)$, as claimed.
\end{proof}

\subsection{Proof of Theorem~\ref{Strebel's}} \label{provaAbels}

If $\phi(\mbA_n(R)) > 0$ for some $n \geq 3$, then $\mbA_n(R)$ and its retract $\Aff(R) \cong { \left( \begin{smallmatrix} * & * \\ 0 & 1 \end{smallmatrix} \right) } \leq \GL_2(R)$ are finitely generated. Thus the abelian group $\Mult(R)$ is finitely generated and $\Addi(R)$ is finitely generated as a $\Z[\Mult(R)]$-module, which shows that $R$ is finitely generated as a ring. This deals with the very first claim of Theorem~\ref{Strebel's} (except possibly when $n = 2$). 

Now, if $\phi(\mbA_2(R)) > 0$, then $\mbA_2(R) \cong \Addi(R)$ is finitely generated as a $\Z$-module. This implies, for every $n \geq 2$, that the unipotent radical $\mbU_n(R)$ of $\mbA_n(R)$ is a finitely generated nilpotent group and thus has $\phi(\mbU_n(R)) = \infty$. (The equality follows e.g. from Lemma~\ref{obviousboundsonphi} by induction on the nilpotency class because all terms of the lower central series of a finitely generated nilpotent group are themselves finitely generated.) 
Moreover, $\mbA_2(R) \cong \Addi(R)$ being finitely generated also implies that the group of units $\Mult(R)$ is finitely generated by Samuel's generalization of Dirichlet's unit theorem~\cite[Section~4.7]{Samuel}. Thus $\mb{D}_n(R)\leq\GL_n(R)$ and $\mb{T}_n(R) \leq \mbA_n(R)$ also have $\phi(\mb{T}_n(R)) = \phi(\mb{D}_n(R)) = \infty$. Since $\mbA_n(R) = \mbU_n(R) \rtimes \mb{T}_n(R)$ and $\mbB_2(R) = \mbU_2(R) \rtimes \mb{D}_2(R)$, it follows from Lemmata~\ref{obviousboundsonphi} and~\ref{BorelGLSL} that $\phi(\mbA_n(R)) = \phi(\BzeroR) = \infty$.
 
 Assume from now on that $R$ is \emph{not} finitely generated as a $\Z$-module. For $n \geq 3$, we first observe that $\phi(\mbA_n(R)) \geq 1$ whenever $\phi(\BzeroR) \geq 1$. Indeed, assume the latter, i.e. $\BzeroR$ is finitely generated. Then, for every $i \in \set{2, \ldots, n-2}$, the subgroups 
 \[
  \mb{E}_{i,i+1}(R) \rtimes (D_i(R) \times D_{i+1}(R)) \cong \mbB_2(R),
 \]
 \[
  \mb{E}_{12}(R) \rtimes D_2(R) \cong \Aff_-(R) \cong {\left( \begin{smallmatrix} 1 & * \\ 0 & * \end{smallmatrix} \right) } \leq \GL_2(R), \, \mbox{ and}
 \]
 \[
  \mb{E}_{n-1,n}(R) \rtimes D_{n-1}(R) \cong \Aff(R) \cong {\left( \begin{smallmatrix} * & * \\ 0 & 1 \end{smallmatrix} \right) } \leq \GL_2(R)
 \]
are also finitely generated by Lemma~\ref{newlemma}. By relations~\eqref{Ancommutators} and~\eqref{SteinbergRelGLn}, the subgroups above generate all of $\mbA_n(R)$, whence $\phi(\mbA_n(R)) \geq 1$. Secondly, it is straightforward that $\mbA_n(R)$---still for $n \geq 3$---retracts e.g. onto $\Aff(R)$, which implies $\phi(\mbA_n(R)) \leq \phi(\BzeroR)$ by Theorem~\ref{apendice}.

From the observations above, the proof of part~\eqref{facil2} will be concluded once we show that $\mbA_3(R)$ can never be finitely presented, i.e. $\phi(\mbA_3(R)) \leq 1$. To see this, in case $\Mult(R)$ is itself finite or is \emph{infinitely} generated, then $\phi(\mbA_3(R)) = \phi(\mbU_3(R)) = 0$. In case $\Mult(R)$ has torsion-free rank at least one and if $\mbA_3(R)$ were finitely presented, then its metabelian quotient $\mbA_3(R)/Z(\mbA_3(R)) = \mbA_3(R) / \mb{E}_{13}(R)$ would also be finitely presented by~\cite[Corollary~5.6]{BieriStrebel}. But the complement of the $\Sigma$-invariant~\cite{BieriStrebel} of the $\Z[\mb{T}_1(R)]$-module $\mbU_3(R) / \mb{E}_{13}(R) \cong \mb{E}_{12}(R) \oplus \mb{E}_{23}(R)$ is readily seen to contain antipodal points, which for us means that $\mbU_3(R) / \mb{E}_{13}(R)$ is not tame as a $\Z[\mb{T}_1(R)]$-module. This would contradict~\cite[Theorem~5.1]{BieriStrebel} and we are done with part~\eqref{facil2}.
 
Turning to part~\eqref{nao-facil} after the previous remarks, it remains to check for $n\geq 4$ that $\phi(\mbA_n(R)) \geq 2$ when $\phi(\BzeroR) \geq 2$. Suppose the latter holds, i.e. $\BzeroR$ is finitely presented. By Lemma~\ref{acaonub}, the group $\mbA_n(R)$ for $n \geq 4$ acts cocompactly and by cell-permuting homeomorphisms on the simplicial complex $CC(\scrH(n,R))$. Since $\phi(\BzeroR) \geq 2$, we know from Lemma~\ref{acaonub} and Proposition~\ref{stabilizersareF2} that the stabilizer in $\mbA_n(R)$ of any cell of $CC(\scrH(n,R))$ is finitely presented. Since $CC(\scrH(n,R))$ (for $n \geq 4$) is connected and simply-connected by Proposition~\ref{Holzs} and Theorem~\ref{AbelsHolzThmCC}, it follows from Theorem~\ref{Brownzinho} that $\mbA_n(R)$ is finitely presented. That is, $\phi(\mbA_n(R)) \geq 2$, as required. \qed

\subsubsection{Remarks on the proof of Theorem~\ref{Strebel's}} \label{lombra}

The author was unable to prove purely geometrically that $CC(\scrH(n,R))$ is simply-connected. The argument given here, whose main technical ingredient is Proposition~\ref{Holzs}, is the only step in the proof of Theorem~\ref{Strebel's} whose methods are similar to those of Strebel's in~\cite{StrebelAbels}. Altogether, there are two key differences between our techniques. 

Assuming $\left\{ \left(\begin{smallmatrix} \diamond & * \\ 0 & 1 \end{smallmatrix}\right) \mid * \in R, \diamond \in Q \right\} \leq \GL_2(R)$ to be finitely presented (and fixing such a presentation), Strebel gives \emph{concrete} generators and relations for his groups $A_n(R,Q)$ for $n \geq 4$. (Recall that $A_n(R,R^\times) = \mbA_n(R)$.) Presentations of $\mbA_n(R)$ using our methods can be extracted using~\cite[Theorem~1]{Brown0}. Alternatively, one can combine the presentation for $\mbU_n(R)$ from Proposition~\ref{Holzs} with a presentation of the torus $\mbT_n(R)$ to construct a presentation for $\mbA_n(R) = \mbU_n(R) \rtimes \mbT_n(R)$. Such presentations, however, are somewhat cumbersome. Thus on the one hand, Strebel's proof has an advantage in that his sets of generators and relations are cleaner.

On the other hand, our proof of Proposition~\ref{Holzs} drawing from Holz's ideas~\cite[Anhang]{Holz} is advantageous in that it suggests a $K$-theoretical phenomenon behind finiteness properties of Abels' groups. It is well-known that classical non-exceptional groups are finitely generated (resp. finitely presented) whenever their ranks are large enough or the base ring has good $K_1$- and $K_2$-groups; cf.~\cite{HahnO'Meara}. For instance, a large rank $n$ gives one enough space in $\GL_n(R)$ to work with elementary matrices via commutator calculus and thus deduce many relations from the standard ones. The same happens with $\mbA_n(R)$---the hypothesis $n \geq 4$ is necessary for positive results, but Holz observes further that one can spare some generators (and some relations) for $\mbA_n(\Z[1/p])$ in the case $n \geq 5$ in comparison to $\mbA_4(\Z[1/p])$. This observation is incorporated in our generalization and is the reason why $CC(\scrH(n,R))$ is $3$-dimensional for $n = 4$ but merely $2$-dimensional for $n \geq 5$. 

\subsection{Conjecture~\ref{EndlichkeitseigenschaftenAbels} in the arithmetic set-up} \label{ExemploConjectura}

We close the paper by spelling out a proof of the following special case of Conjecture~\ref{EndlichkeitseigenschaftenAbels}. Though it has not appeared in this general form in the literature before, we claim no originality---it is a simple combination of famous results mentioned in the introduction.

\begin{pps}
  Let $\OS$ be a Dedekind domain of arithmetic type. If either $\carac(\OS) = 0$ or if $\carac(\OS) > 0$ and $S$ has at most three elements, then the $S$-arithmetic Abels groups $\mbA_n(\OS)$ satisfy Conjecture~\ref{EndlichkeitseigenschaftenAbels}. That is,
  \[
  \phi(\mbA_n(\OS)) = \min\set{n-2, \phi(\Bzero(\OS))} \mbox{ for all } n \geq 2.
 \]
\end{pps}

\begin{proof}
 If $\carac(\OS) = 0$, the Kneser--Tiemeyer local-global principle~\cite[Theorem.~3.1]{Tiemeyer} allows us to assume that $S$ contains a single non-archimedean place. Also, the equality $\phi(\Bzero(\OS)) = \infty$ holds by~\cite[Corollary~4.5]{Tiemeyer}. By restriction of scalars (see e.g.~\cite[Lemma~3.1.4]{Margulis}), it suffices to consider the case where $\Frac(\OS) = \Q$. In this set-up, $\OS$ is of the form $\OS = \Z[1/p]$ for some prime number $p \in \N$. Here $\mbA_n(\Z[1/p])$ has $\phi(\mbA_n(\Z[1/p])) < n-1$, for otherwise it would be of homological type $\FPn{n-1}$ and thus of type $\FPn{\infty}$ by~\cite[Proposition]{BieriZHomology}. In particular, its center $Z(\mbA_n(\Z[1/p]))$ would be finitely generated by~\cite[Corollary~2]{BieriZHomology}. However, $Z(\mbA_n(\Z[1/p]))$ is the elementary subgroup $\mb{E}_{1n}(\Z[1/p]) \cong \Addi(\Z[1/p])$, which is not finitely generated, yielding a contradiction. Since $\phi(\mbA_n(\Z[1/p])) \geq n-2$ by~\cite[Theorem~B]{AbelsBrown} and Brown's criteria (Theorem~\ref{Brownzinho} and~\cite[Proposition~1.1]{Brown}), we obtain
\[
 \phi(\mbA_n(\OS)) = n - 2 = \min\set{n-2, \infty} = \min\set{n-2, \phi(\Bzero(\OS))}.
\]
In case $\carac(\OS) > 0$, we have $\phi(\Bzero(\OS)) = |S|-1$ by Theorem~\ref{Bux's}. Thus, if $S$ has at most three elements, it follows from Theorem~\ref{Strebel's} that
\[
 \phi(\mbA_n(\OS)) = \min\set{0, 1, 2, n-2} = \min\set{n-2, \phi(\Bzero(\OS))}.
\]
\end{proof}

\subsection*{Acknowledgments} 
{\small 
Part of this work grew out of interesting discussions with Herbert~Abels, Stephan~Holz, Benjamin~Br\"uck, and Alastair~Litterick, whom I thank most sincerely. I am indebted to Ralph~Strebel for the mathematical correspondences and for sharing with me his manuscripts, and to my Ph.D. advisor, Kai-Uwe~Bux, for his guidance. The author also thanks the anonymous referees for valuable suggestions.
}

\printbibliography
 
 \end{document}